\documentclass[dvips]{imsart}

\RequirePackage[OT1]{fontenc}
\usepackage{bbm}
\usepackage[numbers]{natbib}
\usepackage{amsthm,amsmath}
\usepackage{graphics,ifthen}
\usepackage{latexsym}
\usepackage{mathrsfs}
\usepackage{amssymb}

\startlocaldefs

\newtheoremstyle{example}{\topsep}{\topsep}%
     {}
     {}
     {\rmfamily}
     {}
     {\newline}
     {\thmname{#1}\thmnumber{ #2}\thmnote{ #3}}

   \theoremstyle{example}

\numberwithin{equation}{section}
\theoremstyle{plain}

\newtheorem{prop}{Proposition}[section]

\newtheorem{rem}{Remark}[section]

\newtheorem{cor}{Corollary}[section]

\newcommand{\Lower}[2]{\smash{\lower #1 \hbox{#2}}}
\newcommand{\ben}{\begin{enumerate}}
\newcommand{\een}{\end{enumerate}}
\newcommand{\bi}{\begin{itemize}}
\newcommand{\ei}{\end{itemize}}

\endlocaldefs

\begin{document}

\begin{frontmatter}
\title{Generalized Mittag Leffler distributions arising as limits in preferential attachment models\protect} \runtitle{Mittag Leffler limits}

\begin{aug}
\author{\fnms{Lancelot F.} \snm{James}\thanksref{t1}\ead[label=e1]{lancelot@ust.hk}}

\thankstext{t1}{Supported in
part by the grant RGC-HKUST 601712 of the HKSAR.}

\runauthor{Lancelot F. James}

\affiliation{Hong Kong University of Science and Technology}

\address[a]{Lancelot F. James\\ The Hong Kong University of Science and
Technology, \\Department of Information Systems, Business Statistics and Operations Management,\\
Clear Water Bay, Kowloon, Hong Kong.\\ \printead{e1}.}

\contributor{James, Lancelot F.}{Hong Kong University of Science
and Technology}

\end{aug}

\begin{abstract}
For $0<\alpha<1,$ and $\theta>-\alpha,$ let $(S^{-\alpha}_{\alpha,\theta+r})_{\{r\ge 0\}}$ denote an increasing(decreasing) sequence of variables forming a time inhomogeneous Markov chain whose marginal distributions are equivalent to generalized Mittag Leffler distributions. 
We exploit the property that such a sequence may be connected with the two parameter $(\alpha,\theta)$ family of 
Poisson Dirichlet distributions with law $\mathrm{PD}(\alpha,\theta)$. We demonstrate that the sequences serve as limits in certain types of preferential attachment models. As one illustrative application, we describe the explicit joint limiting distribution of scaled degree sequences arising under a class of linear weighted preferential attachment models as treated in Mori\cite{Mori}, with weight $\beta>-1.$ 
When $\beta=0$ this corresponds to the Barbasi-Albert preferential attachment model. 
We are in fact primarily interested in distributional properties of $(S^{-\alpha}_{\alpha,\theta+r})_{\{r\ge 0\}}$ and related quantities arising in more intricate exchangeable sampling mechanisms, with direct links to nested mass partitions governed by $\mathrm{PD}(\alpha,\theta).$ We construct sequences of nested $(\alpha,\theta)$ Chinese restaurant partitions of $[n]$. From this, we identify and analyze relevant quantities that may be thought of as mimics for vectors of degree sequences, or differences in tree lengths. We also describe connections to a wide class of continuous time coalescent  processes that can be seen as a variation of stochastic flows of bridges related to generalized Fleming-Viot models.
Under a change of measure our results suggest the possibilities for identification of limiting distributions related to consistent families of nested Gibbs partitions of $[n]$ that would otherwise be difficult by methods using moments or Laplace transforms. In this regard, we focus on special simplifications obtained in the case of $\alpha=1/2.$ That is to say, limits derived from a  $\mathrm{PD}(1/2|t)$ distribution. Throughout we present some distributional results that are relevant to various settings. We close by describing nested schemes varying in $(\alpha,r).$ 
\end{abstract}

\begin{keyword}[class=AMS]
\kwd[Primary ]{60C05, 60G09} \kwd[; secondary ]{60G57,60E99}
\end{keyword}

\begin{keyword}
\kwd{Coagulation-Fragmentation Duality, Mittag Leffler distributions, Preferential attachment models, Pitman-Yor processes, Recursive trees}
\end{keyword}

\end{frontmatter}
\section{Introduction}
For each $0<\alpha<1,$ let $S_{\alpha}$ denote a random variable whose law coincides with a positive stable random variable with index $\alpha$ specified by its Laplace transform $\mathbb{E}[{\mbox e}^{-\omega S_{\alpha}}]={\mbox e}^{-\omega^{\alpha}}$ and density denoted as $f_{\alpha}(t).$
Now define the variables $S_{\alpha,\theta}$ for each $\theta>-\alpha,$ as having a density, denoted by $f_{\alpha,\theta},$ formed by polynomially tilting a stable density as follows
\begin{equation}
f_{\alpha,\theta}(t)=c_{\alpha,\theta}t^{-\theta}f_{\alpha}(t)
\label{PDdiversity}
\end{equation}
where
$
c_{\alpha,\theta}:=\Gamma(\theta+1)/\Gamma(\theta/\alpha+1),
$
and satisfies for $\delta+\theta>-\alpha$
\begin{equation}
\mathbb{E}[S^{-\delta}_{\alpha,\theta}]=\frac{\Gamma(\theta+1)}{\Gamma(\theta/\alpha+1)}
\mathbb{E}[S^{-(\delta+\theta)}_{\alpha}]=\frac{\Gamma(\frac{(\theta+\delta)}{\alpha}+1)}{\Gamma({\theta+\delta}+1)}\frac{\Gamma(\theta+1)}{\Gamma(\theta/\alpha+1)}.
\label{moment}
\end{equation}
The variable $S^{-\alpha}_{\alpha}\overset{d}=S^{-\alpha}_{\alpha,0}$
is often referred to as having a Mittag-Leffler distribution, hence it is natural to consider $S^{-\alpha}_{\alpha,\theta}$ as generalized Mittag-Leffler variables. In terms of combinatorial objects, versions of such variables arise as limits in the two parameter Poisson Dirichlet framework of~\cite{PY97} as follows: 
Let $(P_{k})_{\{k\ge 1\}}$ be the collection of ranked probability masses summing to $1,$ whose law, denoted as $\mathrm{PD}(\alpha,\theta),$ follows a Poisson-Dirichlet distribution with parameters $(\alpha,\theta)$ as described in~Pitman and Yor~\cite{PY97} and Pitman\cite{Pit02,Pit06}. From those works, it follows that letting  $K_{n}$ denote the number of blocks in a $\mathrm{PD}(\alpha,\theta)$ partition of $[n]={\{1,2,\ldots,n\}},$ that is to say the well-known two parameter $(\alpha,\theta)$ Chinese restaurant process, then as 
$n\rightarrow \infty,$  $n^{-\alpha}K_{n}\rightarrow S^{-\alpha}_{\alpha,\theta},$ almost surely. There are of course other known asymptotic results involving the number of blocks of a certain size etc. Following\cite{PPY92,Pit06,PY97},  a version of $S^{-\alpha}_{\alpha,\theta}$ may be interpreted in terms of the local time up to time $1$ of a generalized Bessel process. The following relation shows that $S_{\alpha,\theta}$ is a measurable function of the $(P_{k})\sim \mathrm{PD}(\alpha,\theta)$ which makes sense of conditioning $(P_{k})|S_{\alpha,\theta}=t;$
$
S_{\alpha,\theta}:=\lim_{i\rightarrow \infty}{(i\Gamma(1-\alpha)P_{i})}^{-1/\alpha},$
almost surely. For general $\alpha,$ they also arise in various P\'olya urn and random tree growth models as described in for instance in~\cite{Haas,SvanteUrn}. There are however instances where the limit is not recognized.

We also note that $S^{-1/2}_{1/2,\theta}\overset{d}=2G^{1/2}_{\theta+1/2},$ where $G_{\theta+1/2}$ denotes a $\mathrm{Gamma}(\theta+1/2,1)$ variable. Random variables that are powers of gamma variables have played a key role in recent work by~\cite{Bubeck, Pek2013,Pek2014,Pek2015,PitmanR}. One can surmise that such results for $\alpha=1/2$ can possibly be adapted for the general $\alpha$ case.
Indeed, the works of 
\cite{Chaumont,DevroyeGen, DevroyeJames2, JamesLamperti, JamesPGarxiv, PY97} have shown that $S_{\alpha,\theta}$ satisfies many interesting distributional identities demonstrating a notion of a beta-gamma-stable algebra. However, for example, one important problem considered in Pek\"oz, R\"ollin and Ross \cite{Pek2014} are results related to the scaled limiting distribution of the joint degree distribution of linearly weighted variations of the Barabasi-Albert~\cite{Barabasi} preferential attachment model. This requires quite specific information about the joint distributional behavior between variables in the limit. 
This is considerably more challenging in the general $\alpha$ setting.
Although of interest, our aim in this present setting is not to mimic or adapt the methods in\cite{Pek2013,Pek2014,Pek2015}, but rather to provide more details and indicators on what types of limits might arise in various models, in $n,$ having random limits when scaled by $n^{\alpha}$.

Let $B_{a,b}$ denote a $\mathrm{Beta}(a,b)$ variable. James~\cite{JamesPGarxiv} notes there are versions of the generalized Mittag Leffler variables that satisfy the following exact equality for any $\theta>-\alpha,$
\begin{equation}
S^{-\alpha}_{\alpha,\theta}=S^{-\alpha}_{\alpha,\theta+\alpha}B^{\alpha}_{(\theta+\alpha,1-\alpha)}=S^{-\alpha}_{\alpha,\theta+1}B_{(\frac{\theta+\alpha}{\alpha},\frac{1-\alpha}{\alpha})}
\label{keyID1}
\end{equation}
where $S^{-\alpha}_{\alpha,\theta+\alpha}$ is independent of $B_{(\theta+\alpha,1-\alpha)}$
and $B_{(\frac{\theta+\alpha}{\alpha},\frac{1-\alpha}{\alpha})}$ is independent of $S_{\alpha,\theta+1}.$ [See \cite[eq. (2.11)]{JamesLamperti} for an in distribution version of this result applied to a wider range of parameters].
By recursion, this leads to two sequences forming Markov chains. For an integer $r\ge 0$ an interpretation, in terms of size biased deletion of excursion intervals of certain generalized Bessel bridges, of the first such sequence  $(S^{-\alpha}_{\alpha,\theta+r\alpha})_{r\ge 0\}},$ as is well known, may be read from Perman, Pitman and Yor~\cite[Corollary 3.15]{PPY92}.  Perhaps more simply, the sequence represents a dual Markov chain corresponding to the operation of size biased deletion and insertion as described in Pitman and Yor\cite[Proposition 34 and 35]{PY97}. The sequence 
encodes such operations relative to a nested family of $(\mathrm{PD}(\alpha,\theta+r\alpha))_{r\ge 0}$ distributions. Reading for $r$ increasing describes the states relative to a deletion operation.

Our primary interest is in the second sequence $(S^{-\alpha}_{\alpha,\theta+r})_{\{r\ge 0\}}$ which encodes  Markov chains for the following family of distributions $(\mathrm{PD}(\alpha,\theta+r))_{r\ge 0}.$ Specifically this is encoded by the recursion formed from the equality,
\begin{equation}
S^{-\alpha}_{\alpha,\theta}=S^{-\alpha}_{\alpha,\theta+1}B_{(\frac{\theta+\alpha}{\alpha},\frac{1-\alpha}{\alpha})}.
\label{keyID2}
\end{equation}

The family can be seen to coincide with discrete dual fragmentation coagulation operations described in~\cite{BertoinGoldschmidt2004, Dong2006}, although the particular role of the sequence $(S^{-\alpha}_{\alpha,\theta+r})_{\{r\ge 0\}}$ is not emphasized.
These authors, as mentioned in \cite[Remarks p.1712]{Dong2006}, do cite relations of their constructions to random recursive trees and other trees and graphs constructed under preferential attachment. One could say that $(S^{-\alpha}_{\alpha,\theta+r})_{\{r\ge 0\}}$ is a family of generalized Mittag Leffler distributions under a $\mathrm{PD}(\alpha,\theta)$ discrete coagulation or fragmentation regime. We shall simply refer to $(S^{-\alpha}_{\alpha,\theta+r})_{\{r\ge 0\}}$ as a $\mathrm{PD}(\alpha,\theta)$ sequence. 
The states  $(\mathrm{PD}(\alpha,\theta+r))_{r\ge 0},$ read for $r$ increasing correspond to  fragmentation schemes.

The $\mathrm{PD}(\alpha,1-\alpha)$ case arises in~Haas, Miermont, Pitman and Winkel~\cite{Haas}.
There the sequence $(S^{-\alpha}_{\alpha,r+1-\alpha})_{\{r\ge 0\}}$ is interpreted as increasing lengths of nested families of trees.
The general  Markov chain associated with
 $(S^{-\alpha}_{\alpha,\theta+r})_{\{r\ge 0\}}$ subject to a change of measure, is presented in James~
\cite{JamesPGarxiv} which we shall reproduce here.
One could also deduce this from~\cite[Proposition 18]{Haas}, by a change of measure, since their result involves any $\alpha.$ 
Subsequent to these,  the Markov chain based on constructions in~
Haas and Goldschmidt~\cite{GoldHaas} involves a $\mathrm{PD}(\alpha,\alpha)$ sequence for $\alpha\leq 1/2.$ We note further that the M and L preferential attachment models in~\cite{Pek2014} correspond to the case of $\mathrm{PD}(1/2,0)$ and $\mathrm{PD}(1/2,1/2),$ respectively.
 
Both sets of Markov chains are well defined when conditioned on $S_{\alpha,\theta}=t,$
leading to sequences governed by a $\mathrm{PD}(\alpha|t),   $ law as defined in Pitman~\cite{Pit02,Pit06}. That is if $(P_{k,0})_{\{k\ge 1\}}$ is the mass partition having law $\mathrm{PD}(\alpha,\theta),$ then its conditional distribution is $\mathrm{PD}(\alpha|t),$ the distribution of the families $((P_{k,r})_{k\ge 1})_{\{r\ge 1\}}$ in the Markov chain are then determined by transition rules of known form. The $\mathrm{PD}(1/2|t)$ case has special cancellation properties, which under the regime of the first sequence  $(S^{-1/2}_{1/2,\theta+r/2})_{\{r\ge 0\}},$ 
translates into constructions for the standard additive coalescent and Brownian fragmentation processes in~\cite{AldousPit,BerBrown, Pit02,Pit06}.
Here we shall present some details for the $\mathrm{PD}(1/2|t)$ sequence under the second regime.

For general $\alpha,$ these conditioning arguments now allow one to mix $t$ relative to any non-negative distribution, where the mixing distribution can be expressed as $h(t)f_{\alpha}(t)$ for any non-negative function $h(t),$ such that $\mathbb{E}[S_{\alpha}]=1.$ Thus  we say $(P_{k})_{\{k\ge 1\}}$ has the distribution $\mathrm{PK}_{\alpha}(h\cdot f_{\alpha}),$ if 
$\mathrm{PK}_{\alpha}(h\cdot f_{\alpha})=\int_{0}^{\infty}\mathrm{PD}(\alpha|t)h(t)f_{\alpha}(t)dt.$
This distribution is a Poisson-Kingman distribution based on a stable subordinator with mixing distribution $h(t)f_{\alpha}(t)$ as defined in Pitman~\cite{Pit02,Pit06}. $\mathrm{PD}(\alpha,\theta)$ arises by choosing $h(t)=t^{-\theta}c_{\alpha,\theta}.$ 

\subsection{Outline}
 The paper will now progress as follows. In section 2 we will present a detailed description of the limiting joint degree distribution of preferential attachment models considered by Mori~\cite{Mori} and others. In section 3, we will present a formal description of the pertinent Markov chain under general $\mathrm{PK}_{\alpha}(h\cdot f_{\alpha})$ distributions models. This parallels Pitman, Perman and Yor~\cite[Theorem 2.1]{PPY92} in the $\alpha$-stable setting. Section 4 describes nested families of random partitions of $[n]$ determined by a $\mathrm{PD}(\alpha,\theta)$ sequence of Chinese restaurant processes. We introduce and obtain some distributional results for an interesting class of variables, in $n,$ $(\xi_{n,0},\ldots,\xi_{n,r}).$ These can be seen as mimics for vectors of degree sequences. We further describe some joint limits where the idea is one can then map to various constructions of trees and graphs.  By a change of measure these results can extended to any $\mathrm{PK}_{\alpha}(h\cdot f_{\alpha}).$ Hence, this allows one to describe limits for models based on nested sequences of general Gibbs partitions of $[n],$~\cite{GnedinPitmanI,Pit02,Pit06}. We present relevant calculations for this general setting in section 4.2. We partially view these contributions as helping to provide a blueprint for construction of models having more flexible properties as demonstrated by their limiting distributions. We would add that it seems quite unlikely that one would be able to characterize(recognize) such limits by the usual methods. In Section 5, in terms of practical implementation, we can consider all possibilities in the $\alpha=1/2$ case by obtaining explicit results for $\mathrm{PD}(1/2|t).$  Section 6 describes how to further embed/nest $\mathrm{PD}(\alpha\delta,\theta)$ nested schemes into $\mathrm{PD}(\alpha,\theta)$ nested schemes for any $0<\delta<1,$ which in some sense offers a coalescent version of the results of \cite{CurienHaas}.

\begin{rem} More details and related results in terms of basic properties of $(S^{-\alpha}_{\alpha,\theta+r})_{\{r\ge 0\}}$ are discussed in the unpublished manuscript of James\cite{JamesPGarxiv}. See there for connections to models where  $h(t)f_{\alpha}(t)=[\zeta^{1-1/\alpha}/\alpha] t {\mbox e}^{-t\zeta^{1/\alpha}}{\mbox e}^{\zeta}f_{\alpha}(t).$ The entire range of $\mathrm{PD}(\alpha,\theta),$ for $\theta>-\alpha,$ is obtained by randomizing $\zeta$ to have a $\mathrm{Gamma}(\frac{\theta+\alpha}{\alpha},1)$ distribution.
\end{rem}
\section{The explict joint degree distribution of a class of linearly weighted graph and $\beta$-recursive tree preferential attachment models}
We now describe the limiting joint degree distribution of the linearly weighted preferential attachment graph model obtained in~Mori\cite{Mori}. 
[See Athreya, Ghosh and Sethuraman~\cite{Athreya} for a more general extension]. Bertoin and Uribe-Bravo\cite{BerUribe}, note that the model of \cite{Mori} is equivalent to certain recursive tree models 
as discussed for instance in Devroye~\cite[section 5]{DevroyeBranch}, As such, we lift the description of the scale free tree construction given in \cite{BerUribe}.

Fix $\beta>-1,$ 
and start for $n = 1$ from the unique tree $\mathbb{T}_{1}$ on
$\{0, 1\}$ which has a single edge connecting $0$ and $1$. Then suppose that $\mathbb{T}_{n}$ has been constructed for some $n\ge 1$, and for every $i\in {\{0,1,\ldots,n\}},$ denote by $d_{n}(i)$ the degree
of the vertex $i\in\mathbb{T}_{n}$ Conditionally given $\mathbb{T}_{n}$, the tree 
$\mathbb{T}_{n+1}$ is derived from $\mathbb{T}_{n}$ by incorporating the new vertex $n+1$ and creating an edge between $n+1$ and a vertex
$v_{n}\in\mathbb{T}_{n}$ chosen at random according to the law
$$
\mathbb{P}(v_{n}=i|\mathbb{T}_{n})=\frac{d_{n}(i)+\beta}{2n+\beta(n+1)}{\mbox { for }}i\in {\{0,1,\ldots,n\}}
$$
For reference we shall call these models 
$\beta-$\textit{recursive trees.}

Let $\overset{a.s}\rightarrow$ denote convergence almost surely.
From \cite{Athreya,Mori} one has for any $r\ge 0,$ that as $n\rightarrow \infty$ the  joint vector 
\begin{equation}
{n^{-\frac{1}{(2+\beta)}}}(d_{n}(0),d_{n}(1),\ldots,d_{n}(r))\overset{a.s}\rightarrow (\xi_{0},\xi_{1},\ldots,\xi_{r}),
\label{Morivector}
\end{equation}
where $(\xi_{0},\ldots,\xi_{r})$ has joint moments specified in Mori~\cite{Mori}. Furthermore an important result for the scaled maximal degree is obtained
$$
{n^{-\frac{1}{(2+\beta)}}}\max_{i\ge 0}d_{n}(i)
\overset{a.s}\rightarrow \max_{i\ge 0}\xi_{i}
$$
See also Durrett~\cite{Dgraphs} and van der Hofstad~\cite[Section 8]{Hofstad}. These results correspond to the model of Barabasi and Albert~\cite{Barabasi} when $\beta=0.$

We will now show this model corresponds to \textit{components} of a  $\mathrm{PD}(\alpha,1-2\alpha)$ distribution.  From~(\ref{PDdiversity}) and (\ref{moment}), for each fixed $j\ge 1,$ the density of $S_{\alpha,j-2\alpha}$ is given by $f_{\alpha,j-2\alpha}(t),$ and furthermore 
$$
\mathbb{E}[S^{-k\alpha}_{\alpha,j-2\alpha}]
=\frac{\Gamma(\frac{j-\alpha}{\alpha}+k)\Gamma(j+1-2\alpha)}
{\Gamma(\frac{j-\alpha}{\alpha})\Gamma(j+1-2\alpha+k\alpha)}.
$$
\begin{prop}\label{Moriprop}Set $\beta\alpha=1-2\alpha>-\alpha,$ 
and let $(S^{-\alpha}_{\alpha,r+1-2\alpha})_{\{r\ge 0\}}$ denote the sequence of $\mathrm{PD(}\alpha,1-2\alpha)$ $\alpha$-diversities satisfying the recursive identity
\begin{equation}
S^{-\alpha}_{\alpha,j-2\alpha}=S^{-\alpha}_{\alpha,j+1-2\alpha}B_{j}
\label{Moridegreerecursion}
\end{equation}
where $B_{j}=S^{-\alpha}_{\alpha,j-2\alpha}/S^{-\alpha}_{\alpha,j+1-2\alpha}$ are mutually independent $\mathrm{Beta}(\frac{j-\alpha}{\alpha},\frac{1-\alpha}{\alpha})$ random  variables. Furthermore $(B_{1},\ldots,B_{j})$ is independent of $S^{-\alpha}_{\alpha,\ell-2\alpha}$ for $\ell> j.$ 
\begin{enumerate}
\item[(i)]
Then for every integer $r\ge 0,$ the joint distribution of the sequence 
\begin{equation}
(S^{-\alpha}_{\alpha,1-2\alpha},S^{-\alpha}_{\alpha,2-2\alpha}-S^{-\alpha}_{\alpha,1-2\alpha},\ldots, S^{-\alpha}_{\alpha,r+1-2\alpha}-S^{-\alpha}_{\alpha,r-2\alpha})
\label{Morisdegreesequence}
\end{equation}
 is equivalent in distribution component-wise and jointly to the vector
$$(\xi_{0},\xi_{1},\ldots,\xi_{r})$$
in~(\ref{Morivector}).
\item[(ii)]It follows that, for $S_{\alpha,-2\alpha}:=0.$
$$
{n^{-\alpha}}\max_{i\ge 0}d_{n}(i)
\overset{a.s}\rightarrow \max_{i\ge 0} (S^{-\alpha}_{\alpha,i+1-2\alpha}-S^{-\alpha}_{\alpha,i-2\alpha}).
$$
\item[(iii)]One  may set $\xi_{0}=S^{-\alpha}_{\alpha,1-2\alpha}$ and $\xi_{1}=S^{-\alpha}_{\alpha,2-2\alpha}-S^{-\alpha}_{\alpha,1-2\alpha}.$ Then noting that $B_{1}$ is a symmetric  $\mathrm{Beta}(\frac{1-\alpha}{\alpha},\frac{1-\alpha}{\alpha})$ random variable, there is the distributional identity 
\begin{equation}
\xi_{1}=S^{-\alpha}_{\alpha,2-2\alpha}[1-B_{1}]\overset{d}=
S^{-\alpha}_{\alpha,2-2\alpha}B_{1}\overset{d}=\xi_{0}
\label{degreeid}
\end{equation}
\item[(iv)]
Note for $(S_{\alpha,\theta},S_{\alpha,1+\theta})$ in a $\mathrm{PD}(\alpha,\theta)$ sequence, the correspondence in~(\ref{degreeid})
only holds for the case $\theta=1-2\alpha.$
\end{enumerate}
\end{prop}
\begin{proof}
One can verify [(i)] by checking that the joint moments of the vector in~(\ref{Morisdegreesequence}) correspond to joint moments of $(\xi_{0},\xi_{1},\ldots,\xi_{r})$ provided in \cite{Mori}. However, while true, this is rather tedious. Mori\cite[Lemma 3]{Mori} shows that for each $j,$
$(\xi_0+\ldots+\xi_{j})B_{j}=
(\xi_0+\ldots+\xi_{j-1}),$ where $B_{j}$ has the same beta distribution as in the $\mathrm{PD}(\alpha,1-2\alpha)$ sequence and $(\xi_0+\ldots+\xi_{j})$ is independent of $B_{j}.$ Due to scaling, there are a myriad of potential solutions for the $(\xi_{j}).$ Nonetheless the  recursion in~(\ref{Moridegreerecursion}) establishes the result if one can show that one can set $\xi_{0}=S^{-\alpha}_{\alpha,1-2\alpha}.$ This is true since $\mathbb{E}[S^{-k\alpha}_{\alpha,1-2\alpha}]=  \mathbb{E}[\xi^{k}_{0}].$
\end{proof}

\begin{rem}The equivalence in distribution of $(\xi_{0},\xi_{1})$ in Proposition~\ref{Moriprop} was noted in \cite{Mori}, and is otherwise evident in the description of the $\beta$-recursive tree. The result in~(\ref{degreeid}) shows the only $\mathrm{PD}(\alpha,\theta)$ case we could have considered is $\theta=1-2\alpha.$ 
\end{rem}

\begin{rem} The case $\beta=\infty$ corresponds to $\alpha\rightarrow 0$ which, by continuity, 
$$
\lim_{\alpha\rightarrow 0}\mathrm{PD}(\alpha,1-2\alpha)=\mathrm{PD}(0,1),
$$
yields the Poisson Dirichlet model $\mathrm{PD}(0,1)$. The rates become $\log(n)$ and  $\xi_{j}=1$ for all $j\ge 0.$ 
\end{rem}
\begin{rem} \cite{Pek2014} first provide an explicit description of these limits when $\beta=0,$ that is $\alpha=1/2$ and hence the $\mathrm{PD}(1/2,0).$
\end{rem}

\section{$\mathrm{PK}_{\alpha}(h\cdot f_{\alpha})$ Markov  chains}
We now present the formal details of the Markov chain for $(S^{-\alpha}_{\alpha,\theta+r})_{\{r\ge 0\}}$ under $\mathrm{PD}(\alpha,\theta)$ and under a general change of measure to $\mathrm{PK}_{\alpha}(h\cdot f_{\alpha})$ as described in~\cite{JamesPGarxiv}. As was noted earlier, one can also deduce this from ~\cite{Haas}. It suffices to work with the basic case of a $\mathrm{PD}(\alpha,0)$ sequence. Note from the recursion, there is the identity
\begin{equation}
S^{-\alpha}_{\alpha,0}=S^{-\alpha}_{\alpha,r}\prod_{k=1}^{r}B_{k}
\label{stableid}
\end{equation}
where here $B_{k}$ are independent $\mathrm{Beta}{(\frac{\alpha+k-1}{\alpha},\frac{1-\alpha}{\alpha})}$ variables independent of  $S_{\alpha,r}.$

\begin{prop}\label{JPPY}For each $r,$ let $(T_{\alpha,0}, T_{\alpha,1},\ldots,T_{\alpha,r})$ denote a vector of random variables such that $T_{\alpha,0}\overset{d}=S_{\alpha}$ and there is the relationship for each integer $k$
\begin{equation}
T_{\alpha,(k-1)}=T_{\alpha,k}\times V^{-1/\alpha}_{k}
\label{Vequation}
\end{equation}
where $V_{k}$ has a  $\mathrm{Beta}{(\frac{\alpha+k-1}{\alpha},\frac{1-\alpha}{\alpha})}$ distribution, independent of $T_{\alpha,k}$ and marginally
$T_{\alpha,k}\overset{d}=S_{\alpha,k}.$ Then, the conditional distribution of $T_{\alpha,k}$ given  $T_{\alpha,k-1}=t$ is the same for all $k$ and equates to the density,
\begin{equation}
P(T_{\alpha,1}\in ds|T_{\alpha,0}=t)/ds=\frac{\alpha^{2}}{\Gamma(\frac{1-\alpha}{\alpha})}
\frac{(s/t)^{\alpha-1}(1-(s/t)^{\alpha})^{\frac{(1-\alpha)}{\alpha}-1}f_{\alpha}(s)}{t^{2}f_{\alpha}(t)},
\label{transitionV}
\end{equation}
for $s<t.$
By a change of variable $v=(s/t)^{\alpha}$ the density of $V_{1}|T_{\alpha,0}=t$ is given by
\begin{equation}
P(V_{1}\in dv|T_{\alpha,0}=t)/dv=\frac{\alpha}{\Gamma(\frac{1-\alpha}{\alpha})}
\frac{(1-v)^{\frac{(1-\alpha)}{\alpha}-1}f_{\alpha}(v^{1/\alpha}t)}{tf_{\alpha}(t)}.
\label{transitionV2}
\end{equation}
Furthermore $(V_{1},\ldots,V_{r})$ are independent variables, independent of $T_{\alpha,r}.$
The sequence is a Markov chain, governed by a $\mathrm{PD}(\alpha,0)$ law.
\end{prop}
\begin{proof}
Because of the independence between $V_{k}$ and $T_{\alpha,k}$ the proof just reduces to an elementary Bayes rule argument. Details are presented for clarity. The distribution of $T_{\alpha,k-1}|\hat{T}_{\alpha,k}=s$ is just $V^{-1/\alpha}_{k}s,$ where $V_{k}\sim\mathrm{Beta}{(\frac{\alpha+k-1}{\alpha},\frac{1-\alpha}{\alpha})}$. Use the fact that for each $k,$  $T_{\alpha,k}$ has density
$$
f_{\alpha,k}(s)=\frac{\Gamma(k+1)}{\Gamma(\frac{k+\alpha}{\alpha})}s^{-k}f_{\alpha}(s),
$$
to show that the joint density of $T_{\alpha,k-1},\hat{T}_{\alpha,k})$ is,
\begin{equation}
\frac{\alpha^{2}\Gamma(k)t^{-(k+1)}}{\Gamma(\frac{k+\alpha-1}{\alpha})\Gamma(\frac{1-\alpha}{\alpha})}
(s/t)^{\alpha-1}(1-(s/t)^{\alpha})^{\frac{(1-\alpha)}{\alpha}-1}f_{\alpha}(s).
\label{jointSk}
\end{equation}
 Now divide~(\ref{jointSk}) by the $f_{\alpha,k-1}(t)$ density of $T_{\alpha,k-1},$ to obtain (\ref{transitionV}). The Markov chain is otherwise evident from the exact equality statement.
\end{proof}

\begin{cor}\label{corollary1}
As consequences of Proposition~\ref{JPPY} the distribution of the quantities above with respect to a $PK_{\alpha}(h\cdot f_{\alpha})$ are given by (\ref{transitionV}) and specifying $T_{\alpha,0}$ to have density $h(t)f_{\alpha}(t).$
\begin{enumerate}
\item[(i)]
 In particular, the joint law of $(V_{1},\ldots,V_{r},T_{\alpha,r})$ is given by,
\begin{equation}
\left[\prod_{k=1}^{r}f_{B_{k}}(v_{k})\right]h(s/\prod_{l=1}^{r}v^{1/\alpha}_{l})f_{\alpha,r}(s)ds
\label{jointV}
\end{equation}
where $f_{B_{k}}$ denotes the density of a $\mathrm{Beta}{(\frac{\alpha+k-1}{\alpha},\frac{1-\alpha}{\alpha})}$ variable. $f_{\alpha,r}(s)=c_{\alpha,r}s^{-r}f_{\alpha}(s).$
\item[(ii)]It follows that the conditional distribution of $T_{\alpha,r}|V_{1},\ldots, V_{r}$ is proportional to $h(s/\prod_{l=1}^{r}v^{1/\alpha}_{l})f_{\alpha,r}(s).$
\item[(iii)]Relative to  $\hat{T}_{\alpha,r},$ for each $j=1,2,\ldots,r$
$$
T^{-\alpha}_{\alpha,j-1}=T^{-\alpha}_{\alpha,r}\times \prod_{l=j}^{r}V_{l}
$$
\end{enumerate}
\end{cor}
\begin{rem}\label{Expecth} The fact that the quantity in~(\ref{jointV}) integrates to $1,$ follows from the identity~(\ref{stableid}). Which reads as
$$
\mathbb{E}_{\alpha,0}[h(S_{\alpha})]=\mathbb{E}_{\alpha,0}[h(S_{\alpha,r}\times \prod_{i=1}^{r}B^{-1/\alpha}_{\left(\frac{i-1+\alpha}{\alpha},\frac{1-\alpha}{\alpha}\right)})]=1.
$$
\end{rem}

\begin{rem}\label{randomtrees}Comparing (\ref{transitionV}) with Haas, Miermont, Pitman and Winkel~\cite[Proposition 18, (ii),(iii)]{Haas} shows that under a $\mathrm{PD}(\alpha,1-\alpha)$ model, where for each $k=1,2,\ldots;$ $T_{\alpha,k-1}\overset{d}=S_{\alpha,k-\alpha},$ $T^{-\alpha}_{\alpha,k-1}$ equates to the total length  of $\mathcal{\tilde{R}}^{\mathrm{ord}}_{k},$ say $
\mathbb{D}(\mathcal{\tilde{R}}^{\mathrm{ord}}_{k})=T^{-\alpha}_{\alpha,k-1}.$
Where $\mathcal{\tilde{R}}^{\mathrm{ord}}_{k}$ is a member of an increasing family
$(\mathcal{\tilde{R}}^{\mathrm{ord}}_{k})$
of leaf-labeled $\mathbb{R}$-trees with edge lengths, arising as limits in Ford's sequential construction. It follows from~(\ref{Vequation}) that, in this setting, $(V_{k})$ can be interpreted as
$$
V_{k}=\frac{\mathbb{D}(\mathcal{\tilde{R}}^{\mathrm{ord}}_{k})}{
\mathbb{D}(\mathcal{\tilde{R}}^{\mathrm{ord}}_{k+1})}\overset{d}=B_{(\frac{k}{\alpha},\frac{1-\alpha}{\alpha})},
$$
which is independent of $\mathbb{D}(\mathcal{\tilde{R}}^{\mathrm{ord}}_{k+1})={T}^{-\alpha}_{\alpha,k}\overset{d}=S^{-\alpha}_{\alpha,k+1-\alpha}.$ In fact $(V_{1},\ldots,V_{k})$ are mutually independent and independent of $\mathbb{D}(\mathcal{\tilde{R}}^{\mathrm{ord}}_{k+1}).$  
See~\cite{Haas} for a more precise interpretation of $(\mathcal{\tilde{R}}^{\mathrm{ord}}_{k}).$ See also \cite{Dong2006} for related discussions involving fragmentation by  $\mathrm{PD}(\alpha,1-\alpha)$ models.
\end{rem}
\begin{rem}Note there are other distributions besides $\mathrm{PD}(\alpha,\theta)$ that may produce the same sequences $(S^{-\alpha}_{\alpha,\theta+r})_{\{r\ge 0\}}.$ From a point of view of wide applicability of our results, this is rather fortunate. A key word in our exposition is~\textit{version}. The explicit constructions via bridges in~\cite{JamesPGarxiv} or the analysis of~\cite{Haas}, in the $\mathrm{PD}(\alpha,1-\alpha)$ case, already verifies the existence of the appropraite versions of variables we identify via Corollary~\ref{corollary1} with respect to a sequence of mass partitions $((P_{k,r})_{\{k\ge 1\}})_{\{r\ge 0\}}$ following a sequence of laws determined by $
\mathrm{PK}_{\alpha}(h\cdot f_{\alpha}).$ In the next section we will work with characterizing features of such families. Namely nested random partitions of $[n]
$ derived from the appropriate Chinese restaurant processes.
\end{rem}

\section{$\mathrm{PD}(\alpha,\theta)$ nested  Chinese restaurant processes}\label{nested}
For any fixed $r\ge 0,$ set $(P_{k,r})_{\{k\ge 1\}}\sim \mathrm{PD}(\alpha,\theta+r),$ and independent of this let $(U_{k,r})_{k\ge 1}$ be a collection of iid $\mathrm{Uniform}[0,1]$ random variables.  Then the random probability measure $P_{\alpha,\theta+r}(y):=\sum_{k=1}^{\infty}P_{k,r}\mathbb{I}_{\{U_{k,r}\leq y\}}$ is a $\mathrm{PD}(\alpha,\theta+r)$-bridge,  also known as a Pitman-Yor process as coined in~\cite{IJ2001}. For $j=1,\ldots,r$ one may set $B_{j}=S^{-\alpha}_{\alpha,\theta+j-1}/S^{-\alpha}_{\alpha,\theta+j}$ which are independent $\mathrm{Beta}(\frac{\theta+\alpha+j-1}{\alpha},\frac{1-\alpha}{\alpha})$ variables independent of $P_{\alpha,\theta+r}.$ Let $\mathbb{U}(y)=y \in [0,1]$ denote a $\mathrm{Uniform}[0,1]$ cdf. Now for each $j$ define independent simple bridges
\begin{equation}
\lambda_{j}(y)=B_{j}\mathbb{U}(y)+(1-B_{j})\mathbb{I}_{\{\tilde{U}_{j}\leq y\}}
\label{simplebridges}
\end{equation}
Then the coagulation operation in ~\cite{Dong2006} can be encoded by the compositional identity, for each $r\ge 1,$ 
$$
P_{\alpha,\theta+r-1}(y)=P_{\alpha,\theta+r}(\lambda_{r}(y))=P_{\alpha,\theta+r}(B_{r}y)+P_{\alpha,\theta+r}(1-B_{r})\mathbb{I}_{\{\tilde{U}_{r}\leq y\}}
$$
where $P_{\alpha,\theta+r}(1-B_{r})$ has a $\mathrm{Beta}(1-\alpha,\theta+\alpha+r-1)$ distribution. Furthermore, one can show that $P_{\alpha,\theta+r}(B_{r}y)/P_{\alpha,\theta+r}(B_{r})=P_{\alpha,\theta+r\alpha}(y)$ independent of $P_{\alpha,\theta+r}(B_{r}).$ More generally, for any $r\ge 1,$
$$
P_{\alpha,\theta}(\cdot)=P_{\alpha,\theta+r}\circ \lambda_{r}\circ\cdots\circ\lambda_{1}(\cdot).$$
It follows that if $F^{-1}$ denotes a possibly random quantile function, then for every $r\ge 1$
$$
P^{-1}_{\alpha,\theta}(\cdot)=\lambda^{-1}_{1}\circ\cdots\circ\lambda^{-1}_{r}\circ P^{-1}_{\alpha,\theta+r}(\cdot).
$$
We shall use these properties to construct nested sequences of Chinese restaurant process partitions of $[n]=\{1,2,\ldots,n\}.$ Note a dual fragmentation process can be deduced from~\cite{Dong2006} which will produce nested partitions with the same distributions in reverse order. For a Chinese restaurant process following a $\mathrm{PD}(\alpha,\theta+r)$ distribution, the sampling scheme proceeds as follows. The first customer with index $\{1\}$ is seated to a new table $A_{1,r}.$ After $n$ customers arrive in succession, a partition of $[n],$ $(A_{1,r},\ldots A_{K_{n,r},r})$ where $K_{n,r}\leq n$ are the number of distinct blocks in the partition, and $N_{i,r}=|A_{i,r}|$ are the sizes of each block, is produced. Given this configuration, customer ${\{n+1\}}$ is seated to a new table with probability $(\theta+r+K_{n,r}\alpha)/(\theta+r+n)$  and sits at an existing table $A_{i,r}$ with probability $(N_{i,r}-\alpha)/(\theta+r+n),$ for $i=1,\ldots, K_{n,r}.$ We now describe the combinatorial scheme we have in mind.

\paragraph{Nested $\mathrm{PD}(\alpha,\theta)$ partitions of $[n]$}
\begin{enumerate}
\item[(i)] For any $r\ge 1,$ draw a random partition of $[n],$ $(A_{1,r},\ldots,A_{K_{n,r},r})$ from a $\mathrm{PD}(\alpha,\theta+r)$ Chinese restaurant process scheme.
\item[(ii)] Draw $(U^{*}_{1,r},\ldots,U^{*}_{K_{n,r},r})$ iid Uniform$[0,1]$ variables.
\item[(iii)] Recall that $\tilde{U}_{r}$ is the atom of $\lambda_{r},$ and has a Uniform$[0,1]$ distribution. A $\mathrm{PD}(\alpha, \theta+r-1)$ partition of $[n]$
$(A_{1,r-1},\ldots,A_{{K}_{n,r-1},r-1}),$ is obtained as follows. Blocks of $(A_{1,r},\ldots,A_{K_{n,r},r})$ are merged into a set $A'_{1,r-1}$ defined as
$$
A'_{1,r-1}={\{A_{i,r}:\lambda^{-1}_{r}(U^{*}_{i,r})=\tilde{U}_{r}\}},
$$
if $A'_{1,r-1}$ is not empty, set $A_{1,r-1}=A'_{1,r-1}, $the remaining ${K}_{n,r}-|A'_{1,r-1}|=K_{n,r-1}-1$ blocks of $(A_{1,r},\ldots,A_{K_{n,r},r}),$ are relabeled $A_{2,r-1},\ldots,A_{{K}_{n,r-1},r-1}.$ If $A'_{1,r-1}=\emptyset,$ $K_{n,r-1}=K_{n,r}$
and one sets $A_{k,r-1}=A_{k,r}$ for $k=1,\ldots,K_{n,r}.$
\item[(iv)] Repeat steps [(ii)] and [(iii)] for $r-1,r-2,\ldots,1$ to obtain nested partitions of [n] following $\mathrm{PD}(\alpha,\theta+j)$ marginal distributions for $j=0,\ldots r.$
\end{enumerate}
\begin{rem}Kuba and Panholzer\cite[Proposition 3]{kuba} point out that partitions generated by the $\mathrm{PD}(\alpha,\theta)$ Chinese restaurant can be equally generated by the growth process of generalized plane-oriented recursive trees. As such, some variation of our scheme can be used to produce nested version of such trees. 
\end{rem}
\begin{rem} Note in the general 
$\mathrm{PK}_{\alpha}(h\cdot f_{\alpha})$ setting, one would use simple bridges defined as
 $$
\lambda_{k}(y)=V_{k}\mathbb{U}(y)+(1-V_{k})\mathbb{I}_{\{\tilde{U}_{k}\leq y\}}
$$
for $V_{k}=T^{-\alpha}_{\alpha,k-1}/T^{-\alpha}_{\alpha,k}.$
These are the same entities subject to a change of measure, where generally independence no longer holds. 
\end{rem}
\begin{rem}It is a simple matter to show that 
$P_{\alpha,\theta+r}(\cdot)$ converges almost surely to $\mathbb{U}(\cdot)$ as $r\rightarrow \infty.$ Thus implying that $\lambda_{r}\circ\cdots\circ\lambda_{1}(\cdot)$ converges almost sure to $P_{\alpha,\theta}.$ 
See James~\cite[Section 6.4 and Proposition 6.6]{JamesPGarxiv} for distributional results related to the composition of bridges  
$\lambda_{r}\circ\cdots\circ\lambda_{1}(\cdot).$
\end{rem}
\subsection{Mixed Binomial distributions, $\tilde{p}$-mergers and $\beta$-splitting}
Note the nested scheme described above provides nested versions of all the statistics generally associated with random partitions, and appropriate limits. For brevity we shall only concentrate on results for the sequence of the number of blocks $(K_{n,r})_{\{r\ge 0\}}.$ Let $\mathrm{Bin}(m,p)$ denote a Binomial distribution based on $m$ Bernoulli trials with success probability $p.$ We now describe results for $K_{n,r}.$ The first result is immediate from the description of the nested $\mathrm{PD}(\alpha,\theta)$ scheme. 
\begin{prop}\label{propK}For every $r\ge 1,$ consider the the blocks $(K_{n,0},\ldots, K_{n,r})$ produced by a nested $\mathrm{PD}(\alpha,\theta)$ scheme. It follows that for each $n$ $K_{n,j-1}\leq K_{n,j}$ for $j=1,\ldots r$ with properties;
\begin{enumerate}
\item[(i)]For $j=0,\ldots, r$ the marginal distribution of each $K_{n,j}$ is exactly that of the number of blocks of a $\mathrm{PD}(\alpha,\theta+j)$ partition of $[n[.$ For $j=0,\ldots,r-1$
$$
K_{n,j}=(K_{n,j+1}-|A'_{1,j}|+1)\mathbb{I}_{\{|A'_{1,j}|\ge 2\}}+K_{n,j+1}\mathbb{I}_{\{|A'_{1,j}|\in \{0,1\}\}}
$$
\item[(ii)] For each $j$ the conditional distribution of $|A'_{1,j-1}|$ given $(K_{n,j},B_{j})$ is $\mathrm{Bin}(K_{n,j},1-B_{j})$
\end{enumerate}
\end{prop}
 
 Note that Proposition~\ref{propK} shows that in step[(iii)] of the nested $\mathrm{PD}(\alpha,\theta)$ scheme one is repeatedly performing some sort of $\tilde{p}_{\alpha,\theta+j}=1-B_{j}$ merger in the language of
Berestycki~\cite[p. 69-70]{Berestycki}. That is $\ell$ of the ${\{A_{1,j},\ldots,A_{K_{n,j,j}}\}},$ blocks are said to coalesce if $|A'_{1,j-1}|=\ell\ge 2.$ The next results, which follow from elementary calculations,  describes some more details about the distributions of $|A'_{1,j-1}|$ and random variables  $(\xi_{n,0},\xi_{n,1},\ldots,\xi_{n,r})$ we define as follows.
Set 
\begin{equation}
\xi_{n,0}=K_{n,0}\mathbb{I}_{\{|A'_{1,0}|\ge 2\}}=(K_{n,1}-|A'_{1,0}|+1)\mathbb{I}_{\{|A'_{1,0}|\ge 2\}},
\label{empiricaldegroo0}
\end{equation}
and 
for $j=1,\ldots, r,$ define,
\begin{equation}
\xi_{n,j}:=K_{n,j}-K_{n,j-1}=(|A'_{1,j-1}|-1)\mathbb{I}_{\{|A'_{1,j-1}|\ge 2\}}
\label{empiricaldegree}
\end{equation}
As the notation suggests these are meant to be thought of as mimics for degree sequences.
Write the $\mathrm{Beta}(\frac{1-\alpha}{\alpha},\frac{\theta+\alpha}{\alpha})$ density,
$$
\rho_{\alpha,\theta}(v)=\frac{\Gamma(\frac{1+\theta}{\alpha})}
{\Gamma(\frac{\theta+\alpha}{\alpha})\Gamma(\frac{1-\alpha}{\alpha})}v^{1/\alpha-2}{(1-v)}^{\theta/\alpha}.
$$

\begin{prop}In the $\mathrm{PD}(\alpha,\theta)$ setting of Proposition~\ref{propK}, the general distribution of $|A'_{1,0}|,$ given $K_{n,1}=b$ is a  mixed Binomial distribution $\mathrm{Bin}(b,1-B_{1}),$  where $1-B_{1}$ is a $\mathrm{Beta}(\frac{1-\alpha}{\alpha},\frac{\theta+\alpha}{\alpha})$ 
random variable with density function $\rho_{\alpha,\theta}(v).$ Hence the 
probability mass function of $|A'_{1,0}|,$  is
$$
p_{\alpha,\theta}(\ell|b)={b \choose \ell}\frac{\Gamma(\frac{1+\theta}{\alpha})
\Gamma(\frac{\theta+\alpha}{\alpha}+b-\ell)\Gamma(\frac{1}{\alpha}+\ell-1)}
{\Gamma(\frac{\theta+\alpha}{\alpha})\Gamma(\frac{1-\alpha}{\alpha})
\Gamma(\frac{1+\theta}{\alpha}+b)}
$$
Furthermore the distribution of  $|A'_{1,0}|,$ given  $K_{n,1}=b,|A'_{1,0}|\ge 2,$ is
for $2\leq \ell\leq b,$
\begin{equation}
\lambda_{\alpha,\theta}(\ell|b)=\frac{p_{\alpha,\theta}(\ell|b)}{1-p_{\alpha,\theta}(0|b)-p_{\alpha,\theta}(1|b)}
\label{Lambdakernel}
\end{equation}

The corresponding conditional distributions of the random variable $K_{n,1}-|A'_{1,0}|$ can be expressed as $p^{+}_{\alpha,\theta}(\ell|b):=p_{\alpha,\theta}(b-\ell|b)$ and $
\lambda^{+}_{\alpha,\theta}(\ell|b):=
\lambda_{\alpha,\theta}(b-\ell|b)$ respectively.
Replace $\theta$ with $\theta+j-1$ to obtain corresponding results for $|A'_{1,j-1}|,$ given $K_{n,j}$
\begin{enumerate}
\item[(i)] In the Brownian cases, $\mathrm{PD}(1/2,\theta),$ $\theta>-1/2,$
$$
p_{1/2,\theta}(\ell|b)=\frac{(2\theta+1)(2\theta+b-\ell)!b!}{(2\theta+b+1)!(b-\ell)!}
$$
In particular $p_{1/2,0}(j|b)=1/(b+1)$ is the discrete uniform distribution on ${\{0,\ldots,b\}},$ and
$p_{1/2,1/2}(\ell|b)=2(b+1-\ell)/[(b+1)(b+2)].$
\item[(ii)] In the limiting Dirichlet case, $\mathrm{PD}(0,\theta),$
$$
p_{0,\theta}(\ell|b)={b \choose \ell}p^{\ell}_{\theta}{(1-p_{\theta})}^{b-\ell}
$$
is a proper Binomial distribution with success probability $p_{\theta}=1/(\theta+1),$ for $\theta>0.$
\end{enumerate}
\end{prop}
Note starting from  some $\mathrm{PD}(\alpha,\theta+r)$ it follows that every layer of our $\mathrm{PD}(\alpha,\theta)$ nested scheme produces proper consistent infinitely exchangeable partitions in $[n].$ As such, one may view our scheme, as  discrete time coalescent process based on a sequence of merger rates determined by measures $(\Lambda_{\alpha,\theta+r})_{\{r\ge 0\}},$ where 
from $\lambda_{\alpha,\theta}(\ell|b),$ in~$(\ref{Lambdakernel}),$
$$
\Lambda_{\alpha,\theta}(dv)=\frac{\Gamma(\frac{1+\theta}{\alpha})}
{\Gamma(\frac{\theta+\alpha}{\alpha})\Gamma(\frac{1-\alpha}{\alpha})}v^{1/\alpha}{(1-v)}^{\theta/\alpha}dv.
$$
For a fixed time, these schemes are suggestive of relations to a class of $\Lambda_{\alpha,\theta}$-coalescents where $\Lambda_{\alpha,\theta}$ corresponds to a  $\mathrm{Beta}(\frac{1+\alpha}{\alpha},\frac{\theta+\alpha}{\alpha})$-coalescent. However our models should not be confused with such processes. Rather, our models are also  identified as continuous time coalescent processes by defining as waiting  times $(S^{-\alpha}_{\alpha,\theta+r}-S^{-\alpha}_{\alpha,\theta+r-1},S^{-\alpha}_{\alpha,\theta+r-1}-S^{-\alpha}_{\alpha,\theta+r-2},\ldots)$ in reverse order as written, and viewing our exchangeable bridges $P_{\alpha,\theta+r}(\cdot):=F_{S^{-\alpha}_{\alpha,\theta+r}}(\cdot)$ as, stochastic flows of exchangeable bridges, , $(F_{t},t\ge 0),$  in the spirit of the simple bridge constructions of Bertoin and LeGall~\cite{BerFrag,BerLegall00,BerLegall03} and their connection to Fleming-Viot processes. However, our constructions of the analogous simple bridges $\lambda_{r}$ are quite different as they are based on ratio's of local times, $(B_{r}=S^{-\alpha}_{\alpha,\theta+r-1}/S^{-\alpha}_{\alpha,\theta+r}).$ See remarks in \cite{BertoinGoldschmidt2004} for the $\mathrm{PD}(0,\theta)$ case.
This particular identification of waiting times, is, as we shall explain next, hidden in the literature. Again all such interpretations hold in greater generality under $PK_{\alpha}(h\cdot f_{\alpha}).$ Placing things in reverse increasing order, dual fragmentation processes of partitions of $[n]$ can be obtained by appropriate splitting rules. We shall not attempt to formulate this. However one notes that with respect to mass partitions induced by our schemes there is a natural dual connection to the fragmentation schemes described in \cite[pages 1831-1832]{Haas}. Formally, their Proposition 18 and Corollary 19 describe the same waiting time  structure,
in the form of tree edge lengths $\xi_{0}:=S^{-\alpha}_{\alpha,1-\alpha},$ and 
$\xi_{r}:= S^{-\alpha}_{\alpha,r+1-\alpha}-S^{-\alpha}_{\alpha,r-\alpha},$ for $r\ge 1,$ interpreted in terms of renewal sequences in the $\mathrm{PD}(\alpha,1-\alpha)$ setting. Upon a change of measure, this makes sense for our nested schemes in the general $\mathrm{PK}_{\alpha}(h\cdot f_{\alpha})$ setting. We will describe some more details for general $\alpha$ in section 4.2 and very explicit details in the $\mathrm{PD}(1/2|t)$ case in section 5. We next establish a connection with  Aldous's $\beta$-splitting rule for $\beta>-1.$

\begin{prop}The conditional probability mass functions  $(p_{\alpha,\theta},p^{+}_{\alpha,\theta}),$ and hence also $(\lambda_{\alpha,\theta},\lambda^{+}_{\alpha,\theta}),$ are equivalent if and only if 
$\theta=1-2\alpha.$ That is $B_{1}\overset{d}=(1-B_{1})\sim 
\mathrm{Beta}(\frac{1-\alpha}{\alpha},\frac{1-\alpha}{\alpha}).$
Where for $2\leq \ell\leq b,$
$$
\lambda_{\alpha,1-2\alpha}(\ell|b)\propto p_{\alpha,1-2\alpha}(\ell|b)={b \choose \ell}\frac{\Gamma(\frac{2-2\alpha}{\alpha})
\Gamma(\frac{1-\alpha}{\alpha}+b-\ell)\Gamma(\frac{1}{\alpha}+\ell-1)}
{\Gamma(\frac{1-\alpha}{\alpha})\Gamma(\frac{1-\alpha}{\alpha})
\Gamma(\frac{2-2\alpha}{\alpha}+b)}.
$$
\begin{enumerate}
\item[(i)] Setting $\beta\alpha=1-2\alpha>-\alpha$ it follows that 
$$
p_{\alpha,1-2\alpha}(\ell|b)\propto \tilde{q}^{\mathrm{Aldous}-\beta}_{b}(\ell) {\mbox { for }}  1\leq \ell\leq b-1,
$$
equating to the splitting kernel in the $\beta$-splitting model of Aldous~\cite{AldousClad} for the range $\beta>-1.$ Here we use the notation in~\cite[p.1824]{Haas},
\item[(ii)] Hence,  the Yule model case of $\beta=0$ corresponds to a $\mathrm{PD}(1/2,0)$ model with $p_{1/2,0}(\ell|b)=1/(b+1).$
 The \emph{symmetric random trie} case $\beta\rightarrow\infty$ corresponds to a $\mathrm{PD}(0,1)$ model with $p_{0,1}(\ell|b)={b \choose \ell}{(1/2)}^{b}.$
 In particular, for $b\ge 2,$
 $$
 \tilde{q}^{\mathrm{Aldous}-0}_{b}(\ell)=\frac{1}{b-1}{\mbox { and }} \tilde{q}^{\mathrm{Aldous}-\infty}_{b}(\ell)={b \choose \ell}\frac{1}{2^{b}-2}.
 $$

\end{enumerate}
\end{prop}
\begin{rem}For more on $\Lambda$-coalescents see 
\cite{Berestycki,Pit99,Sagitov}.
\end{rem}
\begin{rem}
See \cite{FordD,Haas,McCullagh} for more on the $\beta$-splitting model and its connection to fragmentation trees.  In addition
See \cite{PitmanWinkel}, \cite[Section 3.3,and Proposition 27]{PitmanWinkel2}, \cite[Proposition 18]{Haas} and \cite[p. 1738]{Dong2006}
which  describe relations to the $\alpha$-model of Ford~\cite{FordD} and the Brownian CRT of Aldous~\cite{AldousCRTI,AldousCRTIII}.
\end{rem}

\subsection{Calculations for general $PK_{\alpha}(h\cdot f_{\alpha})$ nested Gibbs partition schemes}
As we have discussed above, viewing $(S^{-\alpha}_{\alpha,\theta+r})_{\{r\ge 0\}},$
in reverse order, as times where mergers occur gives the formalism to recognize our nested scheme as continuous time coalescent processes  based on a sequence of merger determining measures $(\Lambda_{\alpha,\theta+r})_{\{r\ge 0\}},$ that produce, at each merger time, consistent exchangeable partitions of $[n].$ Conditioning on $S_{\alpha,\theta}=t,$ gives in a distributional,  rather than operational, sense the relevant quantities emerging under  $\mathrm{PD}(\alpha|t).$ schemes. Operationally, since consistent families of Chinese restaurant partitions are produced  under this setting, one carries out the same nested scheme under a $\mathrm{PD}(\alpha|t)$
distribution with time points $T^{-\alpha}_{\alpha,0}=t^{-\alpha}$ and $(T^{-\alpha}_{\alpha,r})_{\{r\ge 0\}}.$ $\mathrm{PK}_{\alpha}(h\cdot f_{\alpha})$ nested schemes proceed in a similar fashion. Thus all producing coalescent schemes which are based on sequences of merger rate determining distributions $(1-V_{k})$ whose distributions generally depend on the number of the blocks considered at the time where mergers are to occur. The exception to this are the $\mathrm{PD}(\alpha,\theta)$ distributions, which represents the only instances where the pairs $(V_{k},K_{n,k})$ are independent. We now describe more explicit details of these distributions.

Recall from ~\cite[Proposition 9]{Pit02}, see also ~\cite{GnedinPitmanI, Pit06}, that the EPPF of a  $\mathrm{PD}(\alpha|t)$ partition of $[n]$ with generic blocks $(A_{1},\ldots,A_{K_{n}}),$ respective block sizes $\textbf{n}_{K_{n}}:=(n_{1},\ldots,n_{K_{n}}),$ and $K_{n}=b$ blocks, can be represented as 
\begin{equation}
p_{\alpha}(n_{1},\ldots,n_{b}|t)=\mathbb{G}^{(n,b)}_{\alpha}(t)\prod_{j=1}^{b}\frac{\Gamma(n_{j}-b)}{
\Gamma(1-\alpha)}
\label{Gibbspartition}
\end{equation}
where $\mathbb{G}^{(n,b)}_{\alpha}(t)={(\Gamma(n-b\alpha)f_{\alpha}(t))}^{-1}\alpha^{b}t^{-n}\int_{0}^{t}f_{\alpha}(t-v)(1-v)^{n-b\alpha-1}dv.$ While the expression for $\mathbb{G}^{(n,b)}_{\alpha}(t),$ seems to be generally intractable, the work of Ho, James and Lau~\cite{HJL} shows that this quantity may  be expressed in terms of special functions corresponding to Fox-H functions in the general $\alpha$ setting and Meijer-G functions in the case of $\alpha$ taking on fractional values. Thus extending the observations of Pitman~\cite{Pit02} for the case of $\alpha=1/2.$ Furthermore, they show that $\mathbb{G}^{(n,b)}_{\alpha}(t)={(\Gamma(n)f_{\alpha}(t))}^{-1}\Gamma(b)\alpha^{b-1}{f}_{\alpha,(n,b)}(t),$ where ${f}_{\alpha,(n,b)}(t)$ is the density of the product of independent random variables,
\begin{equation}
\Sigma_{\alpha,n}(b\alpha)\overset{d}=S_{\alpha,b\alpha}B^{-1}_{b\alpha,n-b\alpha}\overset{d}=
S_{\alpha,n}B^{-1/\alpha}_{(b,\frac{n}{\alpha}-b)},
\label{stablecond}
\end{equation}
which interprets as the conditional density of $S_{\alpha,0}|K_{n}=b$ when the pair $(S_{\alpha,0},K_{n}):=(S_{\alpha,0},K_{n,0})$ are interpreted with respect to a $\mathrm{PD}(\alpha,0)$ distribution. Note the identity  on the right hand side of~(\ref{stablecond}) follows from\cite[eq. (2.11)]{JamesLamperti}[see also \cite{DevroyeJames2}]. This leads to the identity $f_{\alpha}(t)=\sum_{b=1}^{n}\mathbb{P}_{\alpha,0}(K_{n}=b){f}_{\alpha,(n,b)}(t),$ where $\mathbb{P}_{\alpha,0}(K_{n}=b),$ denotes the distribution of $K_{n}$ under $\mathrm{PD}(\alpha,0)$ given in~\cite{Pit06} or otherwise deduced from~(\ref{Gibbspartition}). Hence
\begin{equation}
p_{\alpha}(n_{1},\ldots,n_{b}|s)=\frac{f_{\alpha,(n,b)}(s)}{f_{\alpha}(s)}p_{\alpha,0}(n_{1},\ldots,n_{b})
\label{Gibbspartition2}
\end{equation}
where $p_{\alpha,0}(n_{1},\ldots,n_{b})$ is the EPPF under $\mathrm{PD}(\alpha,0),$ and 
$$
\mathbb{P}_{\alpha}(K_{n,0}=b|S_{\alpha,0}=s)=\frac{f_{\alpha,(n,b)}(s)}{f_{\alpha}(s)}\mathbb{P}_{\alpha,0}(K_{n}=b)
$$

It follows that by integrating (\ref{Gibbspartition}) with respect to $h(s)f_{\alpha}(s),$ that for a general $\mathrm{PK}_{\alpha}(h\cdot f_{\alpha})$ distribution its EPPF, and distribution of $K_{n},$ can be represented as 
$$
p^{(h)}_{\alpha}(\textbf{n}_{b})=W_{n,b}\times p_{\alpha,0}(\textbf{n}_{b}){\mbox { and }}\mathbb{P}^{(h)}_{\alpha}(K_{n}=b)=W_{n,b}\times\mathbb{P}_{\alpha,0}(K_{n}=b)
$$
where, as in \cite[Theorem 4.1]{HJL} suppressing dependence on $h$, $W_{n,b}$ can be expressed as 
$$
W_{n,b}=\mathbb{E}_{\alpha,0}[h(S_{\alpha,0})|K_{n,0}=b]
=\mathbb{E}[h(\Sigma_{\alpha,n}(b\alpha))].
$$
See also Gnedin and Pitman\cite{GnedinPitmanI} for other representations. The next result follows from our discussion.
\begin{prop}\label{mini}Suppose that the law of  $(K_{n,0},T_{\alpha,0})$ is determined by 
$\mathrm{PK}_{\alpha}(h\cdot f_{\alpha}).$ Then, 
\begin{enumerate}
\item[(i)]the joint distribution of $(K_{n,0},T_{\alpha,0})$ may be expressed as 
$$
\mathbb{P}_{\alpha,0}(K_{n}=b)\frac{f_{\alpha,(n,b)}(s)}{f_{\alpha}(s)}\times h(s)f_{\alpha}(s).
$$
\item[(ii)]Hence by Bayes rule the conditional distribution of $T_{\alpha,0}|K_{n,0}=b$ is given by $$
f^{(h)}_{T_{\alpha,0}|K_{n,0}}(t|b)=W^{-1}_{n,b}h(s)f_{\alpha,(n,b)}(s).$$
\item[(iii)]Statement [(ii)] implies that for any integrable function $g$
$$
W_{n,b}\mathbb{E}^{(h)}_{\alpha}[g(T_{\alpha,0})|K_{n,0}=b]=\mathbb{E}_{\alpha,0}[g(S_{\alpha,0})h(S_{\alpha,0})|K_{n,0}=b]
$$
where $\mathbb{E}_{\alpha,0}$ denotes expectation relative to the distribution of $(K_{n,0},S_{\alpha,0})$ defined under a $\mathrm{PD}(\alpha,0)$ distribution. This may be also expressed as
$$
\mathbb{E}^{(h)}_{\alpha}[g(T_{\alpha,0})|K_{n,0}=b]=\mathbb{E}[g(\Sigma_{\alpha,n}(b\alpha))h(\Sigma_{\alpha,n}(b\alpha))]W^{-1}_{n,b}
$$
\end{enumerate}
\end{prop}
\begin{proof}
As we mentioned, the first two statements are just simple consequences of our description of the conditional EPPF. Statement [(iii)] follows by a simple manipulation of $s^{-\theta}f_{\alpha,(n,b)}(s)$.
\end{proof}
We now specialize~\label{mini} to the important $\mathrm{PD}(\alpha,\theta)$ case. This result was originally deduced by~\cite{HJL}.
\begin{prop}\label{mini2}Suppose that the law of  $(K_{n,0},T_{\alpha,0})$ are  determined by 
$\mathrm{PD}(\alpha,\theta)$ case where $h(s)=s^{-\theta}c_{\alpha,\theta}.$ Hence one may set $T_{\alpha,0}=S_{\alpha,\theta}.$ It follows that 
$$
W_{n,b}=c_{\alpha,\theta}\mathbb{E}[\Sigma^{-\theta}_{\alpha,n}(b\alpha)]=\frac{\Gamma(n)\Gamma(\theta+1)\Gamma(\frac{\theta+b\alpha}{\alpha})}{\Gamma(\theta+n)\Gamma(b)\Gamma(\frac{\theta+\alpha}{\alpha})}
$$
and
$W^{-1}_{n,b}s^{-\theta}c_{\alpha,\theta}f_{\alpha,(n,b)}(s)$ corresponds to the conditional density of the random variable $S_{\alpha,\theta}|K_{n,0}=b,$ which is equivalent in distribution to, 
$$
\Sigma_{\alpha,n}(\theta+b\alpha)
\overset{d}=S_{\alpha,\theta+b\alpha}B^{-1}_{\theta+b\alpha,n-b\alpha}\overset{d}=
S_{\alpha,n+\theta}B^{-1/\alpha}_{(\frac{\theta}{\alpha}+b,\frac{n}{\alpha}-b)},
$$
\end{prop}

We now present the main result in this section which follows from these facts. 

\begin{prop}\label{mainGibbs} Consider the variables described in Propoosition~\ref{JPPY} $(T^{-\alpha}_{\alpha,r})_{\{r\ge 0\}}.$ and consider the nested $\mathrm{PD}(\alpha,\theta)$ scheme where mergers and hence consistent partitions of $[n]$ occur at time points 
$S^{-\alpha}_{\alpha,\theta+r},$ for $r$ decreasing. 
Then conditioning on $S_{\alpha,\theta}=t,$ the law of the nested schemes are determined by a $\mathrm{PD}(\alpha|t)$ distribution, and in fact are equivalent in (conditional) distribution to a nested family of partitions that was initially constructed from any $\mathrm{PK}_{\alpha}(h\cdot f_{\alpha})$ distribution. Operationally, under $\mathrm{PD}(\alpha|t)$ the consistent family of nested exchangeable partitions can be generated for any $r\ge 1$ by generating a Chinese restaurant process partition of $[n]$ according to the Poisson Kingman law defined by mixing $\mathrm{PD}(\alpha|s)$ relative to the density of $T_{\alpha,r}|T_{\alpha,0}=t,$ which can be determined from Proposition ~\ref{JPPY}. Denote this density as $f^{(t)}_{\alpha,r}(s)=h_{r}(s|t)f_{\alpha}(s)$ 
and the corresponding Poisson Kingman law as 
$\mathrm{PK}_{\alpha}(f^{(t)}_{\alpha,r})$
This represents the configuration of the partition of $[n]$ at the time $T^{-\alpha}_{\alpha,r}.$ Mergers in the past occur at times 
$(T^{-\alpha}_{r-1},\ldots, T^{-\alpha}_{\alpha,0})|T^{-\alpha}_{\alpha, 0}=t^{-\alpha}$ according to simple bridges constructed from $(V_{r},\ldots,V_{1})|T_{\alpha,0}=t,$ producing partitions of $[n]$ with marginal laws determined by a $\mathrm{PK}_{\alpha}(f^{(t)}_{\alpha,j})$ EPPF for $j=r-1,\ldots,0.$ The distributional properties of the relevant quantities in Proposition 4.1 are now described.
\begin{enumerate}
\item[(i)]For each $r\ge 0,$ the conditional distribution of $K_{n,r}$ given $(T_{\alpha,r},\ldots T_{\alpha,0})$ only depends on $T_{\alpha,r}$ and is given by
$$
\mathbb{P}_{\alpha}(K_{n,r}=b|T_{\alpha,r}=s)=\frac{f_{\alpha,(n,b)}(s)}{f_{\alpha}(s)}\mathbb{P}_{\alpha,0}(K_{n}=b).
$$
\item[(ii)]$|A'_{1,j-1}|$ given $(K_{n,j},V_{j},T_{\alpha,j-1},\ldots T_{\alpha,0})$ has a  $\mathrm{Binomial}(K_{n,j},1-V_{j})$ distribution.
\item[(iii)] For any integer $j,$ the distribution of $(K_{n,j},V_{j})|(T_{\alpha,j-1}=s,\ldots T_{\alpha,0}=t_{0})$ only depends on $T_{\alpha,j-1}=s$ and is the same in form as the case where $j=1,$ with $T_{\alpha,0}=s.$
\item[(iv)] From [(i)] and using (\ref{transitionV2}),  the joint distribution of $(K_{n,1}=b,V_{1}))|T_{\alpha,0}=t$
$$
\mathbb{P}_{\alpha,0}(K_{n}=b)
\frac{f_{\alpha,(n,b)}(tv^{1/\alpha})}{f_{\alpha}(tv^{1/\alpha})}\times \frac{\alpha}{\Gamma(\frac{1-\alpha}{\alpha})}
\frac{(1-v)^{\frac{(1-\alpha)}{\alpha}-1}f_{\alpha}(v^{1/\alpha}t)}{tf_{\alpha}(t)}.
$$
which reduces to 
$$
\mathbb{P}_{\alpha,0}(K_{n}=b)
\frac{\alpha}{\Gamma(\frac{1-\alpha}{\alpha})}
\frac{(1-v)^{\frac{(1-\alpha)}{\alpha}-1}f_{\alpha,(n,b)}(tv^{1/\alpha})}{tf_{\alpha}(t)}.
$$
\item[(v)]Noting, from Proposition \ref{mini2}, that $t^{-1}f_{\alpha,(n,b)}(t)/
\mathbb{E}[\Sigma^{-1}_{\alpha,n}(b\alpha)]
$ is the density of $\Sigma_{\alpha,n}(1+b\alpha)$ it follows that
$$
\frac{\alpha t^{-1}}{1-\alpha}
\int_{0}^{1}{(1-v)^{\frac{(1-\alpha)}{\alpha}-1}f_{\alpha,(n,b)}(tv^{1/\alpha})dv}=\mathbb{E}[\Sigma^{-1}_{\alpha,n}(b\alpha)]
f_{Y_{\alpha,n}(b)}(t)
$$
where $f_{Y_{\alpha,n}(b)}(t)$ denotes the density of the random variable,
$$
Y_{\alpha,n}(b)\overset{d}=B^{-1/\alpha}_{(1,\frac{1-\alpha}{\alpha})}\Sigma_{\alpha,n}(1+b\alpha).
$$
\item[(vi)]$Y_{\alpha,n}(b)$ is the random variable corresponding to the conditional distribution of $S_{\alpha,0}|K_{n,1}=b,$ where $(S_{\alpha,0},K_{n,0},K_{n,1})$ follow a $\mathrm{PD}(\alpha,0)$ distribution, and otherwise $S_{\alpha,0}=S_{\alpha,1}B^{-1/\alpha}_{(1,\frac{1-\alpha}{\alpha})}.$
\item[(vii)]It follows that the conditional distribution of $K_{n,1}|T_{\alpha,0}=t$ is given by
$$
\mathbb{P}(K_{n,1}=b|T_{\alpha,0}=t)=
\mathbb{P}_{\alpha,0}(K_{n}=b)
\frac{(1-\alpha)\mathbb{E}[\Sigma^{-1}_{\alpha,n}(b\alpha)]
f_{Y_{\alpha,n}(b)}(t)}{\Gamma(\frac{1-\alpha}{\alpha})f_{\alpha}(t)}.
$$
\item[(viii)] Hence, the conditional density of $(1-V_{1})|K_{n,1}=b,T_{\alpha,0}=t,$
is given by, for $0<p<1,$
$$
p^{-2}\Lambda_{\alpha}(dp|b,t)=\frac{t^{-1}
\alpha}{(1-\alpha)}
\frac{p^{\frac{(1-\alpha)}{\alpha}-1}f_{\alpha,(n,b)}((1-p)^{1/\alpha}t)}{\mathbb{E}[\Sigma^{-1}_{\alpha,n}(b\alpha)]
f_{Y_{\alpha,n}(b)}(t)}dp,
$$
which is the same as the distribution of $(1-V_{k})|K_{n,k}=b,T_{\alpha,k-1}=t,$ for each $k\ge 1.$
\item[(ix)]It follows that the distribution of $|A'_{1,0}|$ given $K_{n,1}=b,T_{\alpha,0}=t$ is specified by its probability mass function
$$
\mathbb{P}_{\alpha}(|A'_{1,0}|=\ell|K_{n,1}=b,T_{\alpha,0}=t)
={b \choose \ell}
\int_{0}^{1}p^{\ell-2}(1-p)^{b-\ell}\Lambda_{\alpha}(dp|b,t).
$$
\end{enumerate}
\end{prop}
We close this section with the following corollary.
\begin{cor}Suppose that $T_{\alpha,0}$ has density 
$h(t)f_{\alpha}(t)$ then, 
\begin{enumerate}
\item[(i)]the conditional density of $T_{\alpha,0}|K_{n,1}=b$ is 
$$
f_{\alpha,K_{n,1}}(t|b)=\frac{h(t)
f_{Y_{\alpha,n}(b)}(t)}{\mathbb{E}[h(Y_{\alpha,n}(b))]}
$$
\item[(ii)]It follows that the conditional density of $(1-V_{1})|K_{n,1}=b$
is given by, for $0<p<1,$
$$
p^{-2}\Lambda_{\alpha}(dp|b)=\frac{
\alpha p^{\frac{(1-\alpha)}{\alpha}-1}}{(1-\alpha)}
\frac{
\mathbb{E}[h(\Sigma_{\alpha,n}(1+b\alpha)(1-p)^{-1/\alpha})]}{
\mathbb{E}[h(Y_{\alpha,n}(b))]}dp.
$$
\item[(iii)]The conditional density of $(1-V_{1})|K_{n,1}=b$
can also be expressed as, for $0<p<1,$ 
$$
p^{-2}\Lambda_{\alpha}(dp|b)=\frac{
\alpha p^{\frac{(1-\alpha)}{\alpha}-1}}{(1-\alpha)}
\frac{
\mathbb{E}_{\alpha,0}[h(S_{\alpha,1}(1-p)^{-1/\alpha})|K_{n,1}=b]}{
\mathbb{E}_{\alpha,0}[h(S_{\alpha,0})|K_{n,1}=b]}dp.
$$

\item[(iv)] Hence when $h(t)=t^{-\theta}c_{\alpha,\theta},$ corresponding to the $\mathrm{PD}(\alpha,\theta)$ case, the conditional distribution  of $(1-V_{1})|K_{n,1}=b$ has a  $\mathrm{Beta}(\frac{1-\alpha}{\alpha},\frac{\theta+\alpha}{\alpha})$ distribution independent of $K_{n,1}.$
\end{enumerate}
\end{cor}

\subsection{Limits of joint vectors under the nested $\mathrm{PD}(\alpha,\theta)$ scheme}
Our constructions allow one to utilize well-known results to easily establish limit theorems. We shall present one such result here. We again note that these results hold in the more general 
$\mathrm{PK}_{\alpha}(h\cdot f_{\alpha})$ setting

\begin{prop}\label{limit}Let $(S^{-\alpha}_{\alpha,\theta+r})_{\{r\ge 0\}}$ denote the $\mathrm{PD}(\alpha,\theta)$ sequence assocated with the nested scheme. For each $r\ge 0$ let $(K_{n,r}),$ $(\xi_{n,0},\ldots,\xi_{n,r})$ be as previously defined. Then
\begin{enumerate}
\item[(i)] As $n\rightarrow \infty,$ jointly and component-wise, for any $r\ge 0$
$$
n^{-\alpha}(K_{n,0},K_{n,1},\ldots,K_{n,r})\overset{a.s.}\rightarrow (S^{-\alpha}_{\alpha,\theta},\ldots,S^{-\alpha}_{\alpha,\theta+r})
$$ 
\item[(ii)] As $n\rightarrow \infty,$ jointly and component-wise, for any $r\ge 0$
$$
n^{-\alpha}(\xi_{n,0},\xi_{n,1},\ldots,\xi_{n,r})\overset{a.s.}\rightarrow 
(\xi_{0},\xi_{1},\ldots,\xi_{r})
$$
where $\xi_{0}=S^{-\alpha}_{\alpha,\theta},$ and $\xi_{j}=S^{-\alpha}_{\alpha,\theta+j}-S^{-\alpha}_{\alpha,\theta+j-1}$, for $j=1,\ldots r.$
\item[(iii)] One may replace $\xi_{n,0}$ with $K_{n,0}$
\end{enumerate}
\end{prop}
\begin{proof}It suffices to show [(i)] as the arguments for [(ii)] are similar. We first note the fact that $n^{-\alpha}K_{n,r}\overset{a.s.}\rightarrow S^{-\alpha}_{\alpha,\theta+r}.$ We may also consider the $(B_{j})$ to be fixed. Using the notation 
$X\overset{a.s.}\sim Y$ to mean $X/Y\rightarrow 1$ a.s., it follows from Proposition~\ref{propK} that
$K_{n,r-1}\overset{a.s.}\sim Y_{r},$ where $Y_{r}=\sum_{k=1}^{K_{n,r}}b_{k,r},$ for $b_{k,r}$ iid Bernoulli($(B_{r})$) variables. Dividing by $K_{n,r}$ it follows that $Y_{r}=K_{n,r}\bar{Y}_{r},$ hence $n^{-\alpha}(K_{n,r-1},K_{n,r}) \overset{a.s.}\sim  n^{-\alpha}K_{n,r}(\bar{Y}_{r},1)\overset{a.s.}\sim S^{-\alpha}_{\alpha,\theta+r}(B_{r},1).,$ which is $(S^{-\alpha}_{\alpha,\theta+r-1},S^{-\alpha}_{\alpha,\theta+r})$ Continuing in this way it follows that 
$K_{n,j-1}\overset{a.s.}\sim K_{n,r}\prod_{j=1}^{r}\bar{Y}_{j}$ for $j=1,\ldots,r.$ The result is concluded by applying the law of large numbers to a vector of Bernoulli sample means and utilizing the definition of  $(S^{-\alpha}_{\alpha,\theta+r})_{\{r\ge 0\}}.$ 
\end{proof}
By taking $\alpha\rightarrow 0$ and known results for $PD(0,\theta)$ distributions we obtain the following corollary
\begin{cor}Set $\alpha=0$ in~Proposition~\ref{limit}
Then this yields results for a  $\mathrm{PD}(0,\theta)$ nested scheme as follows:
\begin{enumerate}
\item[(i)] As $n\rightarrow \infty,$ jointly and component-wise, for any $r\ge 0$
$$
(K_{n,0},K_{n,1},\ldots,K_{n,r})\overset{a.s.}\sim \log(n)(\theta,\theta+1,\ldots,\theta+r)
$$ 
\item[(ii)] As $n\rightarrow \infty,$ jointly and component-wise, for any $r\ge 0$
$$
(\xi_{n,0},\xi_{n,1},\ldots,\xi_{n,r})\overset{a.s.}\sim \log(n)
(\theta,1,\ldots,1).
$$
\end{enumerate}
\end{cor}
\section{The $\mathrm{PD}(1/2|t)$ case}
We now describe all possible limits under any scheme that is asymptotically equivalent to our nested schemes based on $PK_{1/2}(h\cdot f_{1/2})$ distributions.  This is done by looking at their common generator under the $\mathrm{PD}(1/2|t)$ model. Note again that if $K_{n}$ is the number of distinct blocks in a partition of $[n]$ following a $\mathrm{PD}(1/2|t)$ distribution then 
$n^{-1/2}K_{n}\overset{a.s.}\sim 1/\sqrt{t}.$ 
See Pitman~\cite[Section 8]{Pit02} for a nice treatment in regards to the Brownian excursion partition. The simple properties of the $\mathrm{PD}(1/2|t)$ distribution related to the Markov chain in~Perman,Pitman and Yor\cite{PPY92} are also exploited in~\cite{AldousPit, BerBrown,Pit06}. What we are doing here is expressing properties of this distribution relative to operations corresponding to
Proposition~\ref{JPPY}.   
Recall that 
$S_{1/2,0}\overset{d}=S_{1/2}\overset{d}=1/(4G_{1/2}),$ where $G_{1/2}$ is a $\mathrm{Gamma}(1/2,1)$ variable. This means that 
$$
f_{1/2}(t)=\frac{1}{2\Gamma(1/2)}t^{-3/2}{\mbox e}^{-1/4t}
$$
Throughout, let  $(\textbf{e}_{k})$ denote an iid sequence of exponential(1) variables. 

\begin{prop}\label{Browniancase}Consider the setting in Proposition~\ref{JPPY} then the joint distribution of $(T_{1/2,k})_{\{k\ge 1\}}|T_{1/2,0}=t$ 
is described as follows. 
\begin{enumerate}
\item[(i)]Setting $\alpha=1/2$ in~(\ref{transitionV}), 
for each $k\ge 1,$ the conditional density of $T_{1/2,k}|T_{1/2,k-1}=t_{k-1}$ is given by
$$
f_{T_{1/2,k}|T_{1/2,k-1}}(t_{k}|t_{k-1})=\frac{1}{4}t^{-2}_{k}{\mbox e}^{-\frac{1}{4t_{k}}}{\mbox e}^{\frac{1}{4t_{k-1}}} {\mbox { for }}t_{k}<t_{k-1}.
$$ 
\item[(ii)]It follows that $(T_{1/2,k})_{\{k\ge 1\}}|T_{1/2,0}=t$ correspond to the points of an inhomogeneous Poisson point process with intensity
\begin{equation}
\tau(s|t)=\frac{1}{4}s^{-2}\mathbb{I}_{\{s<t\}}
\label{intensity1}
\end{equation}
\item[(iii)]Hence, conditional on $T_{1/2,0}=t,$ for each $k\ge 1,$ one can set
$$
T_{1/2,k}=\frac{1}{4\sum_{\ell=1}^{k}\textbf{e}_{\ell}+1/t}
$$
\item[(iv)]From~(\ref{intensity1}) it follows that 
$2^{-1/2}(T^{-1/2}_{1/2,k})_{\{k\ge 1\}}|T_{1/2,0}=t$  are the points of an inhomogeneous Poisson point process with intensity
\begin{equation}
\rho(y|t)=y\mathbb{I}_{\{y>1/\sqrt{2t}\}}
\label{intensity2}
\end{equation}
\end{enumerate}
\end{prop}

\begin{proof}The results follow by the simplifications unique to the $\alpha=1/2$ case via the density $f_{1/2}(t)$ as applied to the transition density in~(\ref{transitionV}). 
\end{proof}

\begin{rem}
Setting $C_{1}=1/\sqrt{2t},$  $C_{k+1}=1/\sqrt{2T_{1/2,k}},$ for $k\ge 1,$ the 
spacings $C_{k+1}-C_{k}$ can be interpreted as the distribution of the edge lengths in Aldous's\cite[p.278]{AldousCRTIII} construction of the Brownian Random Tree (CRT) when conditioned on $C_{1}=1/\sqrt{2t}.$ As is well known Aldous's construction corresponds to randomizing in the $\mathrm{PD}(1/2,1/2)$ case. However, using our explicit descriptions in Propositions~\ref{Browniancase}, such a construction makes sense for any choice of $t,$ yielding different trees.
\end{rem}

We next describe some implications of Proposition~\ref{Browniancase}.

\begin{prop}\label{PDBrown}For $\alpha=1/2,$ consider the setting in Proposition~\ref{JPPY}. The initial mass partition $(P_{k,0})|T_{1/2,0}=t$ follows a $\mathrm{PD}(1/2|t)$ distribution. A description of the corresponding variables is provided as follows.  Given, $T_{1/2,0}=t.$
\begin{enumerate}
\item[(i)]$T^{-1/2}_{1/2,1}=\sqrt{4\textbf{e}_{1}+1/t}$ and for $k\ge 2,$
$$
T^{-1/2}_{1/2,k}=\sqrt{4\sum_{\ell=1}^{k}\textbf{e}_{\ell}+1/t}
$$
\item[(ii)]$V_{1}=1/\sqrt{4t\textbf{e}_{1}+1},$  with density, for $0<u<1,$
\begin{equation}
f_{V_{1}}(u|t)=\frac{1}{2t}u^{-3}{\mbox e}^{-\frac{(1-u^{2})}{4t u^{2}}}
\label{denV}
\end{equation}
\item[(iii)]For $k\ge 2,$
$$
V_{k}=\frac{\sqrt{4\sum_{\ell=1}^{k-1}\textbf{e}_{\ell}+1/t}}{\sqrt{4\sum_{\ell=1}^{k}\textbf{e}_{\ell}+1/t}}
$$
\item[(iv)]For every $r\ge 0,$ there exist random vectors $(\xi_{n,0},\xi_{n,1},\ldots,\xi_{n,r})$ defined under a $\mathrm{PD}(1/2|t)$ partition scheme, as in (\ref{empiricaldegree}) and (\ref{empiricaldegroo0}) such that, 
$$
n^{-1/2}(\xi_{n,0},\xi_{n,1},\ldots,\xi_{n,r})\overset{a.s.}\sim 
(\xi_{0},\xi_{1},\ldots,\xi_{r})
$$
where $\xi_{0}=1/\sqrt{t},$ 
$\xi_{1}=\sqrt{4\textbf{e}_{1}+1/t}-1/\sqrt{t},$ and for $k\ge 2,$
$$
\xi_{k}=\sqrt{4\sum_{\ell=1}^{k}\textbf{e}_{\ell}+1/t}-
\sqrt{4\sum_{\ell=1}^{k-1}\textbf{e}_{\ell}+1/t}
$$
\item[(v)]Results for any 
$\mathrm{PK}_{1/2}(h\cdot f_{1/2})$ distribution are obtained by randomizing $t$ with respect to $h(t)f_{1/2}(t)$ which can be any non-negative distribution.
\end{enumerate}
\end{prop}

\begin{rem}
Notice that $\xi_{n,0}$ is not fixed, rather $n^{-1/2}\xi_{n,0}\overset{a.s.}\sim 1/\sqrt{t}.$ 
\end{rem}

\begin{rem}
It is now easy to see that the limits for the  M and L preferential attachment models in~\cite{Pek2014} correspond to the case of $\mathrm{PD}(1/2,0)$ and $\mathrm{PD}(1/2,1/2),$ respectively.
\end{rem}
With respect to the calculations in section 4.2, it follows that
$$
\Sigma^{-1}_{1/2,n}(1+b/2)\overset{d}=4G_{n+3/2}B^{2}_{(2+b,2n-b)}\overset{d}=4G_{(b+3)/2}B_{((2+b)/2,(2n-b)/2)}
$$
and 
$$
Y^{-1}_{1/2,n}(b)\overset{d}=B^{2}_{(1,1)}\Sigma^{-1}_{1/2,n}(1+b/2)\overset{d}=B_{(1/2,1)}\Sigma^{-1}_{1/2,n}(1+b/2).
$$
From~\cite[section 3.1]{HJL} we have that 
$\mathbb{G}^{(n,k)}_{1/2}(t)={(\Gamma(n)f_{1/2}(t))}^{-1}\Gamma(k)2^{-k+1}{f}_{1/2,(n,k)}(t),$
is expressible in terms of a  ratio of Meijer G functions as
$$
\frac{G^{0,2}_{2,1} \left(4t
 \left|
\begin{matrix} -\frac{1+k}{2},-\frac{k}{2}
\\ \\-n \end{matrix} \right.
\right)}{G^{0,1}_{1,0} \left(4t
 \left|
\begin{matrix} -\frac{1}{2}
\\ \\\overline{\hspace*{0.2in}} \end{matrix}
\right. \right)}=
\frac{(4t)^{-\frac{1+k}{2}-1}e^{-\frac{1}{4t}}
U\left(-\frac{k}{2}+n, \frac{3}{2}, \frac{1}{4t}\right)}
{(4t)^{-\frac{3}{2}}e^{-\frac{1}{4t}}},
$$
where $U(a,b,c)$ is the confluent hypergeometric function of the
second kind~(see~\cite[p.263]{Lebedev72}). The above ratio reduces
to
$$
2^{-k+1} t^{-\frac{k}{2}+\frac{1}{2}}
U\left(-\frac{k}{2}-\frac{1}{2}+n, \frac{1}{2},
\frac{1}{4t}\right)
$$
via an application of the recurrence
relation~\cite[p.505]{Slater65}
$$
U(a,b,z) = z^{1-b} U(1+a-b,2-b,z).
$$
A change of variable $t = \frac{1}{2} \lambda^{-2}$ yields
the expression $2^{n-k} \lambda^{k-1} h_{k+1-2n}(\lambda)$
inside equation~(110) in Pitman\cite{Pit02}, where
$h_{\nu}(\lambda)$ is the Hermite function of index
$\nu$~\cite[Sec. 10.2]{Lebedev72}, based on the following
relationship,
$$
h_{\nu}(\lambda) = 2^{\nu/2} U\left(-\frac{\nu}{2},
\frac{1}{2}, \frac{\lambda^2}{2}\right).
$$

Hence from Proposition~\ref{mainGibbs} it follows that the joint distribution of $(K_{n,1}=k,V_{1}))|T_{1/2,0}=t,$ can be expressed as, 
$$
\mathbb{P}_{1/2,0}(K_{n}=k)
\frac{t^{-\frac{(k+1)}{2}}\Gamma(n)}{2\Gamma(k)}
U\left(-\frac{k}{2}-\frac{1}{2}+n, \frac{1}{2},
\frac{1}{4tv^{2}}\right).
$$
where, as in Pitman~\cite{Pit02},
$$
\mathbb{P}_{1/2,0}(K_{n}=k)={2n-k-1 \choose n}2^{k+1-2n}.
$$
\section{Nesting across $\alpha$}
Recently, within the context of $\kappa$-stable trees, for $1<\kappa\leq 2,$ Curien and Haas~\cite{CurienHaas}, showed that one can construct  all the stable trees simultaneously as a nested family. In particular for $1<\kappa<\kappa'\leq 2,$ they relate a $\kappa$-stable tree with a $\kappa'$-stable tree rescaled by an independent Mittag-Leffler type distribution. We suspect that hidden in such a story are operations  induced by a $\mathrm{PD}(\alpha,-\alpha\delta)-\mathrm{Frag}$ operator as described in Pitman~\cite[Theorem 12]{Pit99}, for some choice of  $0\leq \delta<1$ and $\theta>-\alpha\delta.$ Rather than pursue that, we describe how one can produce nested $\mathrm{PD}(\alpha\delta,\theta)$ partitions from  $\mathrm{PD}(\alpha,\theta)$ nested schemes via the corresponding
$\mathrm{PD}(\delta,\frac{\theta}{\alpha})-\mathrm{COAG}$ operator. See also \cite[Corollary 10]{Haas2}.
\begin{prop}\label{nestedprop}
Suppose that a sequence of nested local times $(S^{-\alpha}_{\alpha,\theta+r})_{\{r\ge 0\}}$ correspond to the states of a Markov chain $(\mathrm{PD}(\alpha,\theta+r))_{\{r\ge 0\}}$ produced by a $\mathrm{PD}(\alpha,\theta)$  nested CRP scheme as described in section~\ref{nested}.
Let $((Q_{k,r})_{\{k\ge 1\}})_{\{r\ge 0\}}$ denote a nested sequence of mass partitions following precisely a nested sequence of laws 
$(\mathrm{PD}(\delta,\frac{\theta+r}{\alpha}))_{\{r\ge 0\}}$
whose relations will be described more formally below, and where for each fixed $r,$  
$(Q_{k,r})_{\{k\ge 1\}}$ is independent of the corresponding $\mathrm{PD}(\alpha,\theta+r)$ mass partition. Let $P_{\delta,\frac{\theta+r}{\alpha}}$ denote corresponding bridges. A 
$\mathrm{PD}(\alpha\delta,\theta)$  nested sequence of partitions of $[n]$ with $(S^{-\alpha\delta}_{\alpha\delta,\theta+r})_{\{r\ge 0\}},$ corresponding to the states of a Markov Chain $(\mathrm{PD}(\alpha\delta,\theta+r))_{\{r\ge 0\}},$
can be obtained in a distributional sense from the $\mathrm{PD}(\alpha,\theta)$ nested CRP scheme by 
using the coagulation operation of Pitman~\cite{Pit99} encoded in the composition of independent bridges 
$
P_{\alpha\delta,\theta}=P_{\alpha,\theta}\circ P_{\delta,\frac{\theta}{\alpha}},$ where $P_{\delta,\frac{\theta}{\alpha}},$ corresponds to the  $\mathrm{PD}(\delta,\frac{\theta}{\alpha})-\mathrm{COAG}$ operator. More precisely this is encoded in the ordered operation of coagulations described by the relations
$$
P^{-1}_{\alpha\delta,\theta}=P^{-1}_{\delta,\frac{\theta}{\alpha}}\circ P^{-1}_{\alpha,\theta}=
P^{-1}_{\delta,\frac{\theta}{\alpha}}\circ \lambda^{-1}_{\alpha,\theta}\circ\cdots\circ \lambda^{-1}_{\alpha,\theta+r-1}\circ P^{-1}_{\alpha,\theta+r}
$$
where the simple bridge 
$\lambda_{\alpha,\theta}(y):=\lambda_{1}(y),$ as in (\ref{simplebridges}). A scheme that directly relates each pair $(S_{\alpha,\theta+r},S_{\alpha\delta,\theta+r}),$ and hence $(\mathrm{PD}(\alpha,\theta+r), \mathrm{PD}(\alpha\delta,\theta+r)),$ for each step $r,$ can be deduced from the following results. 
\begin{enumerate}
\item[(i)]There is the identity
$$
\lambda_{\alpha,\theta}\circ P_{\delta,\frac{\theta}{\alpha}}(\cdot)=P_{\delta,\frac{1+\theta}{\alpha}}\circ \lambda_{\alpha\delta,\theta}(\cdot)
$$
\item[(ii)]Statement [(i)] means that $\lambda_{\alpha,\theta}\circ P_{\delta,\frac{\theta}{\alpha}}(y)$ is equivalent to the bridge
$$
B_{(\frac{\theta+\alpha\delta}{\alpha},\frac{1-\alpha\delta}{\alpha})}P_{\delta,\frac{\theta}{\alpha}+\delta}(y)+(1-B_{(\frac{\theta+\alpha\delta}{\alpha},\frac{1-\alpha\delta}{\alpha})})\mathbb{I}_{\{U'_{1}\leq y\}}
$$
where $U'_{1}=P^{-1}_{\delta,\frac{\theta}{\alpha}}(\tilde{U}_{1})$ is the Uniform$[0,1]$ variable associated with the first size biased pick from a $\mathrm{PD}(\delta,\theta/\alpha)$ distribution, and is further equivalent to the bridge representation
$$
P_{\delta,\frac{1+\theta}{\alpha}}(
B_{(\frac{\theta+\alpha\delta}{\alpha\delta},\frac{1-\alpha\delta}{\alpha\delta})}\mathbb{U}(y)+(1-B_{(\frac{\theta+\alpha\delta}{\alpha\delta},\frac{1-\alpha\delta}{\alpha\delta})})\mathbb{I}_{\{U'_{1}\leq y\}})=P_{\delta,\frac{1+\theta}{\alpha}}(\lambda_{\alpha\delta,\theta}(y))
$$

\item[(iii)]The equivalences are then encoded by their local times or $\delta$-diversities,
$$
S^{-\delta}_{\delta,\frac{\theta}{\alpha}+\delta}B^{\delta}_{(\frac{\theta+\alpha\delta}{\alpha},\frac{1-\alpha\delta}{\alpha})}=
S^{-\delta}_{\delta,\frac{1+\theta}{\alpha}}B_{(\frac{\theta+\alpha\delta}{\alpha\delta},\frac{1-\alpha\delta}{\alpha\delta})}
$$
\item[(iv)]The identity $P_{\alpha\delta,\theta}(\cdot)=P_{\alpha,\theta}\circ P_{\delta,\frac{\theta}{\alpha}}(\cdot)=P_{\alpha,1+\theta}\circ P_{\delta,\frac{1+\theta}{\alpha}}\circ \lambda_{\alpha\delta,\theta}(\cdot)$ shows that the $\alpha\delta$-diversity of  $P_{\alpha\delta,\theta}$ can be represented as, 
$$
S^{-\alpha\delta}_{\alpha\delta,\theta}=S^{-\alpha\delta}_{\alpha,\theta}S^{-\delta}_{\delta,\frac{\theta}{\alpha}}=S^{-\alpha\delta}_{\alpha,1+\theta}
S^{-\delta}_{\delta,\frac{1+\theta}{\alpha}}B_{(\frac{\theta+\alpha\delta}{\alpha\delta},\frac{1-\alpha\delta}{\alpha\delta})}
=S^{-\alpha\delta}_{\alpha\delta,1+\theta}B_{(\frac{\theta+\alpha\delta}{\alpha\delta},\frac{1-\alpha\delta}{\alpha\delta})}
$$
\item[(v)] Replacing $\theta$ with $\theta+r$ establishes relationships between $(S^{-\alpha}_{\alpha,\theta+r})_{\{r\ge 0\}}$ and $(S^{-\alpha\delta}_{\alpha\delta,\theta+r})_{\{r\ge 0\}}:=(S^{-\alpha\delta}_{\alpha,\theta+r}\times S^{-\delta}_{\delta,\frac{\theta+r}{\alpha}})_{\{r\ge 0\}}$ in the respective nested schemes.
Which is encoded by,
$$
P^{-1}_{\alpha\delta,\theta}=P^{-1}_{\delta,\frac{\theta}{\alpha}}\circ P^{-1}_{\alpha,\theta}=
\lambda^{-1}_{\alpha\delta,\theta}\circ\cdots\circ \lambda^{-1}_{\alpha\delta,\theta+r-1}\circ P^{-1}_{\delta,\frac{\theta+r}{\alpha}}\circ P^{-1}_{\alpha,\theta+r}
$$
where $P^{-1}_{\alpha\delta,\theta+r}=P^{-1}_{\delta,\frac{\theta+r}{\alpha}}\circ P^{-1}_{\alpha,\theta+r},$ and for the simple bridge
$\lambda_{\alpha\delta,\theta+r-1},$ 
$$ 
B_{(\frac{\theta+r-1+\alpha\delta}{\alpha\delta},\frac{1-\alpha\delta}{\alpha\delta})}:=\frac{S^{-\alpha\delta}_{\alpha,\theta+r-1}S^{-\delta}_{\delta,\frac{\theta+r-1}{\alpha}}}{S^{-\alpha\delta}_{\alpha,\theta+r}S^{-\delta}_{\delta,\frac{\theta+r}{\alpha}}}
$$
for $r\ge 1.$ Note,  $P^{-1}_{\delta,\frac{\theta+r-1}{\alpha}}\circ 
\lambda^{-1}_{\alpha,\theta+r-1}
=\lambda^{-1}_{\alpha\delta,\theta+r-1}\circ P^{-1}_{\delta,\frac{\theta+r}{\alpha}},$ 
which follows from [(i)].
\item[(vi)] The relationship between the nested sequence of $\mathrm{COAG}$-operators, 
following  the laws 
$(\mathrm{PD}(\delta,\frac{\theta+r}{\alpha}))_{\{r\ge 0\}}$ is encoded by the local time relations
$$
S^{-\delta}_{\delta,\frac{\theta}{\alpha}}=
S^{-\delta}_{\delta,\frac{1+\theta}{\alpha}}
B^{-\delta}_{(\frac{\theta+\alpha}{\alpha},\frac{1-\alpha}{\alpha})}
B_{(\frac{\theta+\alpha\delta}{\alpha\delta},\frac{1-\alpha\delta}{\alpha\delta})},
$$
substituting generally $\theta$ with $\theta+r-1.$
This follows from[(iv)] by using the identity
$S^{-\alpha\delta}_{\alpha,1+\theta}=S^{-\alpha\delta}_{\alpha,\theta}B^{-\delta}_{(\frac{\theta+\alpha}{\alpha},\frac{1-\alpha}{\alpha})},$ where the variables on the right are not independent. 
\end{enumerate}
\end{prop}
\begin{proof}
Statement [(i)] is a just a consequence of Pitman's coagulation operation combined with the simpler coagulation operator in~\cite{Dong2006} as follows. 
$P_{\alpha,\delta}=P_{\alpha,\theta}\circ P_{\delta,\frac{\theta}{\alpha}}=P_{\alpha,1+\theta}\circ \lambda_{\alpha,\theta}\circ P_{\delta,\frac{\theta}{\alpha}}.$ Also by~\cite{Pit99} $P_{\alpha,1+\theta}\circ P_{\delta,\frac{1+\theta}{\alpha}}=P_{\alpha\delta,1+\theta}.$
So for each $y$ there is the identity of distribution functions 
$$
P_{\alpha\delta,\theta}(y)=P_{\alpha,1+\theta}(\lambda_{\alpha,\theta}( P_{\delta,\frac{\theta}{\alpha}}(y)))=P_{\alpha,1+\theta}(P_{\delta,\frac{1+\theta}{\alpha}}(\lambda_{\alpha\delta,\theta}(y))),
$$
which establishes [(i)]. Although equivalent, statement [(ii)] does not directly appeal to [(i)], and offers a result by direct construction. By definition,  $\lambda_{\alpha,\theta}\circ P_{\delta,\frac{\theta}{\alpha}}(y)$ is given by
$$
B_{(\frac{\theta+\alpha}{\alpha},\frac{1-\alpha}{\alpha})}P_{\delta,\frac{\theta}{\alpha}}(y)+(1-B_{(\frac{\theta+\alpha}{\alpha},\frac{1-\alpha}{\alpha})})\mathbb{I}_{\{P^{-1}_{\delta,\frac{\theta}{\alpha}}(\tilde{U}_{1})\leq y\}}
$$
where $U'_{1}=P^{-1}_{\delta,\frac{\theta}{\alpha}}(\tilde{U}_{1})$ is, as noted, the point associated with the first size biased pick from $\mathrm{PD}(\delta, \frac{\theta}{\alpha}),$ say $1-W_{1}$ following a $\mathrm{Beta}(1-\delta,\frac{\theta}{\alpha}+\delta)$ distribution Use  $U'_{1}=P^{-1}_{\delta,\frac{\theta}{\alpha}}(\tilde{U}_{1}),$ and otherwise replace $P_{\delta,\frac{\theta}{\alpha}}$ with its stick-breaking representation $P_{\delta,\frac{\theta}{\alpha}}(y)=W_{1}P_{\delta,\frac{\theta}{\alpha}+\delta}(y)+(1-W_{1})\mathbb{I}_{\{U'_{1}\leq y\}}.$ Combining terms, one sees that for $B=
B_{(\frac{\theta+\alpha}{\alpha},\frac{1-\alpha}{\alpha})},$ it remains to evaluate the distribution of $BW_{1}$ and $(1-W_{1})B+(1-B)=1-BW_{1}.$ It follows by the usual Beta Gamma algebra that $BW_{1}=B_{{(\frac{\theta+\alpha\delta}{\alpha},\frac{1-\alpha\delta}{\alpha})}}.$ The remaining equivalence is apparent either by direct evalution or an appeal to \cite[Proposition 21]{PY97}. Statement [(iii)] follows  from this and the remaining statements have been explained within the text of the Proposition. 
\end{proof}

We close by noting some connections to the Bolthausen-Sznitman coalescent~\cite{Bolt}.
\begin{cor} Consider the dynamics of the 
Bolthausen-Sznitman coalescent, or $U$-coalescent, as described in for instance \cite[Theorem 14, Corollary 15]{Pit99} or \cite{BerFrag,BerLegall00}. Then setting 
$\alpha={\mbox e}^{-s}, \delta={\mbox e}^{-(t-s)}$ and $\theta=0$ in Proposition~\ref{nestedprop}, establishes various relations between the $U$-coalescent and corresponding nested CRP schemes.
\end{cor}

\begin{rem}This section is partially influenced by recent interactions with Anton Wakolbinger and Martin M\"ohle, whom I thank for their time. One sees that $(S^{e^{-t}}_{(e^{-t},0)},t\ge 0)$ is a version of M\"ohle's~\cite{MohleAlea}
Mittag Leffler process. Setting $(K_{n}(t):=N^{(n)}_{t},t\ge 0)$ to match with the notation in~\cite{MohleAlea}, denoting for fixed $t,$ the number of blocks of a $\mathrm{PD}(e^{-t},0)$ partition of $[n],$ one has for $n\rightarrow \infty,$
$$
\frac{K_{n}(t)}{n^{e^{-t}}}=\frac{N^{(n)}(t)}{n^{e^{-t}}}\overset{a.s.}\sim S^{e^{-t}}_{(e^{-t},0)},$$
as was established in~\cite{MohleAlea}, and otherwise corresponds to results for the $e^{-t}$-diversity.
\end{rem}


\begin{thebibliography}{0}
\bibitem{AldousCRTI}
\textsc{Aldous, D.} (1991).
The Continuum Random Tree. I
\textit{Ann. Probab.} \textbf{19}, 1-28.

\bibitem{AldousCRTIII}
\textsc{Aldous, D.} (1993).
The Continuum Random Tree. III
\textit{Ann. Probab.} \textbf{21}, 248-289.

\bibitem{AldousClad}
\textsc{Aldous, D.} (1995).
Probability Distributions on Cladograms. In: Random Discrete Structures, eds. D. Aldous and R. Pemantle, (1995), 1-18. Springer: IMA Volumes in Mathematics and its Applications 76.

\bibitem{AldousPit}
\textsc{Aldous, D. and Pitman, J.} (1998).
The standard additive coalescent
\textit{Ann. Probab.} \textbf{26}, 1703-1726.

\bibitem{Athreya}
\textsc{Athreya, K. B., Ghosh, A. P.} and \textsc{Sethuraman, S.} (2008). Growth of preferential attachment
random graphs via continuous-time branching processes. \emph{Proc. Indian Acad. Sci. Math. Sci.} \textbf{118}
473?494.

\bibitem{Barabasi}
\textsc{Barabasi, A.L.} and \textsc{Albert, R.} (1999). Emergence of scaling in random networks. \emph{Science} 286
509?512.


\bibitem{Berestycki}
\textsc{Berestycki, N.} (2009).
Recent progress in Coalescent theory.
\emph{Ensaios Matematicos}, Vol. 16, 1-193.

\bibitem{BerBrown}
\textsc{Bertoin, J.} (2000).
A fragmentation process connected to Brownian motion.
\emph{Probability
Theory and Related Fields}, \textbf{117}, 289-301.


\bibitem{BerFrag}
\textsc{Bertoin, J.} (2006). \emph{Random fragmentation and
coagulation processes}, Cambridge University Press.

\bibitem{BertoinGoldschmidt2004} \textsc{Bertoin, J.} and \textsc{Goldschmidt, C.} (2004). Dual random
fragmentation and coagulation and an application to the genealogy
of Yule processes. In \textit{Mathematics and computer science
III: Algorithms, Trees, Combinatorics and Probabilities}, M.
Drmota, P. Flajolet, D. Gardy, B. Gittenberger (editors), pp.
295--308. Trends Math., Birkh\"auser, Basel.

\bibitem{BerLegall00}
\textsc{Bertoin, J. and Le Gall, J.-F.} (2000). The
Bolthausen-Sznitman coalescent and the genealogy of
continuous-state branching processes. \emph{Probab. Theory Related
Fields} \textbf{117} (2000),  249--266.

\bibitem{BerLegall03}
\textsc{Bertoin, J. and Le Gall, J.-F.} (2003). Stochastic flows
associated to coalescent processes. \emph{Probab. Theory Related
Fields} \textbf{126}, 261--288.

\bibitem{BerUribe}
\textsc{Bertoin, J. and Uribe Bravo, G.} (2015)
Supercritical percolation on large scale-free random trees
\textit{Ann. Appl. Probab.}\textbf{25} 81-130.

\bibitem{Bolt}
\textsc{Bolthausen, E. and Sznitman, A.-S.} (1998). On Ruelle’s probability cascades and an abstract
cavity method. \emph{Comm. Math. Phys.} \textbf{197} 247–276.

\bibitem{Bubeck}
\textsc{Bubeck, S.,Mossel, E.} and \textsc{R\'acz, M.Z.} (2015). On the influence of the seed graph in
the preferential attachment model. \emph{IEEE Transactions on Network Science and Engineering,}
\textbf{2} 30?39.

\bibitem{Chaumont}
\textsc{Chaumont, L. and Yor, M.} (2003). \textit{Exercises in
probability. A guided tour from measure theory to random
processes, via conditioning.} Cambridge Series in Statistical and
Probabilistic Mathematics, 13, Cambridge University Press.

\bibitem{CurienHaas}
\textsc{Curien, N., and Haas, B.} (2013). The stable trees are nested. \textit{Probability Theory and Related Fields,} \textbf{157}, 847-883.


\bibitem{DevroyeBranch}
\textsc{Devroye, L.}, (1998). Branching processes and their applications in the analysis of tree structures
and tree algorithms. In \emph{Probabilistic Methods for Algorithmic Discrete Mathematics. Algorithms
Combin.} \textbf{16} 249?314. Springer, Berlin

\bibitem{DevroyeGen}
\textsc{Devroye, L.} (2009). Random variate generation for exponentially and polynomially tilted stable distributions. {\it ACM Transactions on Modeling and Computer Simulation (TOMACS)}\textbf{19}, Issue 4, Article No. 18.

\bibitem{DevroyeJames2}
\textsc{Devroye, L.}, and
\textsc{James, L.} (2014). On simulation and properties of the stable law. \textit{Statistical methods and applications}, \textbf{23}, 307-343.

\bibitem{Dong2006} \textsc{Dong, R.}, \textsc{Goldschmidt, C.} and \textsc{Martin, J.}(2006).
Coagulation-fragmentation duality, Poisson-Dirichlet
distributions and random recursive trees. {\it Ann. Appl. Probab.}
\textbf{16} 1733-1750.

\bibitem{Dgraphs}
\textsc{Durrett, R.} (2007). \emph{Random graph dynamics} (Vol. 200, No. 7). Cambridge: Cambridge university press.


\bibitem{FordD}
\textsc{Ford, D.J.}
Probabilities on cladograms: introduction to the alpha model.
arXiv:math/0511246 [math.PR]

\bibitem{GnedinPitmanI}
\textsc{Gnedin, A. and Pitman, J.} (2005).
Exchangeable Gibbs partitions and Stirling triangles. Zap. Nauchn. Sem. S.-Peterburg. Otdel. Mat. Inst. Steklov. (POMI) 325 (2005), Teor. Predst. Din. Sist. Komb. i Algoritm. Metody. \textbf{12}, 83--102, 244--245; translation in J. Math. Sci. (N. Y.) \textbf{138} (2006), no. 3, 5674--5685.


\bibitem{GoldHaas}
\textsc{Goldschmidt, C.} and \textsc{Haas, B.}. (2015). A line-breaking construction of the stable trees.
\emph{Electronic Journal of Probability}, \textbf{20} 1-24.

\bibitem{Haas}
\textsc{Haas, B., Miermont, G., Pitman, J. and Winkel, M.} (2008).
Continuum tree asymptotics of discrete fragmentations and
applications to phylogenetic models. \emph{Ann. Probab.}
\textbf{36} 1790--1837.

\bibitem{Haas2}
\textsc{Haas, B., Pitman, J. and Winkel, M.} (2009). Spinal
partitions and invariance under re-rooting of continuum random
trees \emph{Ann. Probab.} \textbf{37} 1381--1411.


\bibitem{HJL}
\textsc{Ho, M.W.,  James, L.F} and \textsc{Lau, J.W.} (2007).
Gibbs Partitions (EPPF's) Derived From a Stable Subordinator are Fox H and Meijer G Transforms. arXiv:0708.0619[math.PR]


\bibitem{IJ2001}
\textsc{Ishwaran, H.} and \textsc{James. L.F.} (2001). Gibbs
sampling methods for stick-breaking priors. \textit{J. Amer.
Statist. Assoc.} \textbf{96}, 161--173.

\bibitem{JamesLamperti}
\textsc{James, L.F.}(2010). Lamperti type laws.
\textit{Ann. Appl. Probab.}\textbf{20} 1303-1340.

\bibitem{JamesPGarxiv}
\textsc{James, L.F.} (2013). Stick-breaking $\mathrm{PG}(\alpha,\zeta)$-Generalized Gamma Processes.
Unpublished manuscript.  arXiv:1308.6570[math.PR].

\bibitem{SvanteUrn}
\textsc{Janson, S.}(2006) Limit theorems for triangular urn schemes. \emph{Probability Theory and Related
Fields}, \textbf{134} 417-452.

\bibitem{kuba}
\textsc{Kuba, M.} and \textsc{Panholzer, A.}(2014)
On moment sequences and mixed Poisson distributions.  arXiv:1403.2712 [math.CO].

\bibitem{Lebedev72}
\textsc{Lebedev, N.~N.} (1972). \textit{Special Functions
and Their Applications}, Dover Publications, Incorporated,
New York.

\bibitem{McCullagh}
\textsc{McCullagh, P., Pitman, J.} and \textsc{Winkel, M.} (2008).
Gibbs fragmentation trees. \emph{Bernoulli} \textbf{14} 988-1002.

\bibitem{MohleAlea}
\textsc{M\"ohle, M} (2015).
The Mittag Leffler process and a scaling limit for
the block counting process of the Bolthausen
Sznitman coalescent. \emph{ALEA Lat. Am. J. Probab. Math. Stat.} \textbf{12} 35-53.

\bibitem{Mori}
\textsc{M\'ori, T. F.} (2005). The maximum degree of the Barab\'asi-Albert random tree. \emph{Combinatorics, Probability and Computing}, \textbf{14}, 339-348.



\bibitem{Pek2013}
\textsc{Pek\"oz, E., R\"ollin, A.} and \textsc{Ross, N.} (2013). Degree asymptotics with rates for preferential
attachment random graphs. The \emph{Annals of Applied Probability}, \textbf{23} 1188-1218.


\bibitem{Pek2014}
\textsc{Pek\"oz, E., R\"ollin, A.} and \textsc{Ross, N.} (2014). Joint degree distributions of preferential attachment
random graphs. Preprint. Available at arxiv.org/abs/1402.4686.

\bibitem{Pek2015}
\textsc{Pek\"oz, E., R\"ollin, A.} and \textsc{Ross, N.} (2015)
Generalized gamma approximation with rates
for urns, walks and trees. The \emph{Annals of Probability}, to appear.


\bibitem{PPY92}
\textsc{Perman, M., Pitman, J.} and \textsc{Yor, M.} (1992).
Size-biased sampling of Poisson point processes and excursions.
\emph{Probab. Theory Related Fields.} \textbf{92}, 21-39.

\bibitem{Pit99} \textsc{Pitman, J.} (1999). Coalescents with multiple collisions. \textit{Ann.
Probab.} \textbf{27} 1870--1902.

\bibitem{Pit02}
\textsc{Pitman, J.} (2003). Poisson-Kingman partitions. In
\emph{Science and Statistics: A Festschrift for Terry Speed.}
(D.R. Goldstein, Ed.), 1--34, Institute of Mathematical Statistics
Hayward, California.

\bibitem{Pit06}
\textsc{Pitman, J.} (2006). \textit{Combinatorial stochastic
processes.} Lectures from the 32nd Summer School on Probability
Theory held in Saint-Flour, July 7--24, 2002. With a foreword by
Jean Picard. Lecture Notes in Mathematics, 1875. Springer-Verlag,
Berlin.

\bibitem{PitmanR}
\textsc{Pitman,J.} and \textsc{Racz, M.Z.}(2015)
Beta-gamma tail asymptotics. 
arXiv:1509.02583 


\bibitem{PitmanWinkel}
\textsc{Pitman, J. and Winkel, M.} (2009).
Regenerative tree growth: binary self-similar continuum random trees and Poisson-Dirichlet compositions.
\textit{Ann. Probab.} \textbf{37}, 1999-2041.

\bibitem{PitmanWinkel2}
\textsc{Pitman, J. and Winkel, M.} (2015).
Regenerative tree growth: Markovian embedding of fragmenters, bifurcators and bead splitting processes
\emph{Ann. Probab.} \textbf{43} 2611-2646.

\bibitem{PY92}
\textsc{Pitman, J., and  Yor, M.}(1992). Arcsine laws and interval
partitions derived from a stable subordinator. \emph{Proc. London
Math. Soc.} \textbf{65} 326-356.

\bibitem{PY97}
\textsc{Pitman, J.} and \textsc{Yor, M.} (1997). The two-parameter
Poisson-Dirichlet distribution derived from a stable subordinator.
\textit{Ann. Probab.} \textbf{25}, 855--900.


\bibitem{Sagitov}
\textsc{Sagitov, S.}(1999). The general coalescent with asynchronous mergers of ancestral
lines.  \emph{J. Appl. Prob.} \textbf{36}, 1116–1125.

\bibitem{Slater65}
\textsc{Slater, L.~J.} (1965). Confluent Hypergeometric
Functions. In \textit{Handbook of Mathematical Functions,
With Formulas, Graphs, and Mathematical Tables}
(Abramowitz, M. and Stegun, I.A., eds.), 503--535, Dover
Publications, Incorporated, New York.

\bibitem{Hofstad}
\textsc{van der Hofstad, R.} (2013). Random graphs and complex  networks.May 2013
http://www.win.tue.nl/~rhofstad/NotesRGCN.pdf

\end{thebibliography}
\end{document}